\documentclass[draft, onecolumn]{IEEEtran}
\usepackage{amsmath,amsfonts,amssymb,euscript, graphicx, 
epsfig,
enumerate,float,afterpage, subfigure, ifthen}%

\newtheorem{thm}{Theorem}

\newtheorem{lem}{Lemma}
\newtheorem{defn}{Definition}

\newcommand{\expect}[1]{\mathbb{E}\left\{#1\right\}}
\newcommand{\defequiv}{\mbox{\raisebox{-.3ex}{$\overset{\vartriangle}{=}$}}}

\newcommand{\norm}[1]{||{#1}||}

\newcommand{\bv}[1]{{\boldsymbol{#1} }}
\newcommand{\script}[1]{{{\cal{#1} }}}

\begin{document}

\title
  {Max Weight Learning Algorithms with Application to Scheduling in Unknown Environments}
\author{Michael J. Neely \\ University of Southern California\\ http://www-rcf.usc.edu/$\sim$mjneely\\
\thanks{Michael J. Neely is with the  Electrical Engineering department at the University
of Southern California, Los Angles, CA (web: http://www-rcf.usc.edu/$\sim$mjneely).} 
\thanks{This material is supported in part  by one or more of 
the following: the DARPA IT-MANET program
grant W911NF-07-0028, the NSF grant OCE 0520324, 
the NSF Career grant CCF-0747525.}}

\markboth{}{Neely}

\maketitle

\begin{abstract} 
We consider a discrete time stochastic queueing system where a controller makes a 2-stage
decision every slot.  The decision at the first stage reveals a hidden source of randomness
with a control-dependent (but unknown) probability distribution.  
The decision at the second stage incurs a penalty vector that depends on this revealed 
randomness.  The goal is to stabilize all queues and 
minimize a convex function of the time average penalty vector subject to an
additional set of time average penalty constraints.  This setting fits a wide class of stochastic 
optimization problems.  This  includes problems of opportunistic 
scheduling in wireless networks, where a 2-stage decision about channel measurement and
packet  transmission must be made every slot without knowledge of the underlying 
transmission success probabilities.   
We develop a simple max-weight algorithm that  learns efficient behavior by averaging
functionals of previous outcomes. 
The algorithm yields performance that can be pushed arbitrarily close to optimal, with 
a  tradeoff in convergence time and delay.
\end{abstract} 

\section{Introduction}

We consider a stochastic queueing system that operates in discrete time with unit timeslots
$t \in \{0, 1, 2, \ldots\}$.  Every slot $t$, a controller makes a 2-stage control decision that
affects queue dynamics and 
incurs a random penalty vector.  Specifically, the controller first chooses an action $k(t)$
from a finite set of $K$ ``stage-1'' control actions, given 
by an action set $\script{K} = \{1, \ldots, K\}$.  After the action $k(t) \in \script{K}$ is chosen, 
a random vector $\bv{\omega}(t)$ is revealed, which represents a collection of system 
parameters for slot $t$ (such as channel states for a wireless system).   
The random 
vectors $\bv{\omega}(t)$ are conditionally i.i.d. with distribution function $F_k(\bv{\omega})$ 
over all slots for which $k(t) = k$, where $F_k(\bv{\omega})$ is defined: 
\begin{equation} \label{eq:F-def} 
F_k(\bv{\omega}) \defequiv Pr[ \bv{\omega}(t) \leq \bv{\omega}\left|\right. k(t) = k] \: \: \mbox{ for $k \in \script{K}$}
\end{equation} 
where vector inequality is taken entrywise. 
However, the distribution functions $F_k(\bv{\omega})$ are unknown. 
Based on knowledge of the revealed $\bv{\omega}(t)$
vector, the controller makes an additional decision $I(t)$, where $I(t)$ is chosen from some abstract
(possibly infinite)  set $\script{I}$.  This decision  affects the service rates and arrival processes of the 
queues on slot $t$, and additionally incurs an $M$-dimensional 
\emph{penalty vector} $\bv{x}(t)= (x_1(t), \ldots, x_M(t))$, 
where each entry $m \in \{1, \ldots, M\}$  is a function of $I(t)$, $k(t)$, and $\bv{\omega}(t)$ according
to known functions $\hat{x}_m(k(t), \bv{\omega}(t), I(t))$: 
\begin{equation} \label{eq:penalty-function} 
 x_m(t) = \hat{x}_m(k(t), \bv{\omega}(t), I(t))  \mbox{ for $m \in \{1, \ldots, M\}$} 
 \end{equation} 

The penalties can be either positive, zero, or negative (negative penalties can be used to represent
\emph{rewards}).  Let $\overline{\bv{x}}$ be the  \emph{time average penalty vector} that results from the 
control actions made over time (assuming temporarily that this time average is well defined). 
The goal is to develop a control policy that 
minimizes a convex function $f(\overline{\bv{x}})$ of the time average penalty vector, 
subject to queue stability and  
to an additional set of $N$ linear constraints of the type $h_n(\overline{\bv{x}}) \leq b_n$ for $n \in \{1, \ldots, N\}$,
where the constants $b_n$ are given and the functions  
$h_n(\bv{x})$ are linear over $\bv{x} \in \mathbb{R}^M$.\footnote{For simplicity we treat the case
of linear $h_n(\bv{x})$ functions here, although the analysis can be extended to treat convex
(possibly non-linear) $h_n(\bv{x})$ functions, as considered  in \cite{now} for the case without
``stage 1'' control decisions.  See also Remark 1 in Section \ref{subsection:performance} for a further discussion.} 
This objective is similar to the objectives treated in  
\cite{now} \cite{stolyar-greedy} \cite{neely-fairness-infocom05}  
for stochastic network optimization problems, and the problem can be addressed
using the techniques given there in the following special cases: 
\begin{itemize} 
\item (Special Case 1) There is no ``stage-1'' control action
$k(t)$, so that the revealed randomness $\bv{\omega}(t)$ does not depend on any control decision.
\item  (Special Case 2) The distribution functions $F_k(\bv{\omega})$ are known. 
\end{itemize}

An example of Special Case 1 is the problem 
of minimizing  time average power expenditure in a multi-user 
wireless downlink (or uplink) with random 
time-varying channel states that are known at the beginning of 
every slot.   
Simple max-weight transmission policies are known to solve such problems, even without knowledge of the
probability distributions for the channels or packet arrivals \cite{neely-energy-it}.
An example of Special Case 2 is the same system  with the additional assumption 
that there is a cost to measuring channels at the beginning
of each slot. In this example, we have the option of either measuring the channels (and thus having the 
hidden random channel states revealed to us) or transmitting blindly.   Such a problem is treated
in \cite{chih-ping-channel-measure}, and a related problem with partial channel measurement
is treated in \cite{sanjay-channel-measure}.  Both \cite{chih-ping-channel-measure} and
\cite{sanjay-channel-measure}  solve the problem via 
max-weight algorithms that include  an expectation with respect to the known joint channel state distribution.   
While it is reasonable
to estimate the joint channel state distribution when channels are independent and/or when
the number of channels $M$ is small (and 
the number of possible states per channel is also small),   such estimation 
becomes intractable in cases when channels are correlated and 
there are, say, $1024$ possible states per 
channel (and hence there are $1024^M$ probabilities to be estimated 
in the joint channel state distribution).

Another important example is that of dynamic packet routing and transmission scheduling in 
a multi-commodity, multi-hop network with probabilistic channel errors and multi-receiver diversity.  
The Diversity Backpressure Routing (DIVBAR) algorithm of \cite{neely-divbar-journal}
reduces this problem to a 2-stage max-weight problem where each node decides 
which of the $K$ commodities  
to transmit at the first stage.  After transmission, the random vector of neighbor successes
is revealed, and the ``stage-2'' packet forwarding decision is made. 
If there is a single commodity ($K=1$), 
the problem of maximizing throughput reduces to a problem without ``stage-1'' decisions, 
while if there is more than one commodity
the solution given in \cite{neely-divbar-journal} requires knowledge of the joint 
transmission success probabilities for all neighboring nodes.   It is of considerable interest to 
design a modified algorithm that does not require such probability information.

In this paper, we provide a framework for solving such problems without having a-priori knowledge
of the underlying probability distributions.  
For simplicity, we focus primarily on 1-hop networks, although the 
techniques extend to multi-hop networks using the techniques of \cite{now} \cite{neely-asilomar08}. 
Our approach uses the observation that, rather than requiring an estimate of the  full
probability distributions, all that is
needed is an estimate of a set of  expected \emph{max-weight functionals} that depend on
these distributions.  These can be efficiently estimated using 
penalties incurred on previous transmissions to learn optimal behavior.  


Related stochastic network optimization problems (without the 2-stage decision and
learning component) appear in \cite{vijay-allerton02} \cite{now} \cite{neely-fairness-infocom05} \cite{stolyar-greedy}. 
Work in \cite{vijay-allerton02} considers optimization of a utility function of time average
throughput in an opportunistic scheduling scenario but without queues or stability
constraints.  Work in \cite{now} \cite{neely-fairness-infocom05} 
treats joint queue stability and performance optimization using 
Lyapunov optimization, and work in \cite{stolyar-greedy} treats similar problems in
a fluid limit sense using primal-dual methods.  Sequential channel probing techniques
via dynamic programming are treated in \cite{chaporkar-probing}  \cite{javidi-myopic} \cite{liu-channel-probe}.
General methods for Q-learning, based
on approximate dynamic programming, are presented in \cite{bertsekas-nueral}.  Our approach 
is different and  is based on simpler Lyapunov optimization techniques, 
which, due to the special structure of the problem,  
provide strong (polynomial) bounds on convergence even for high dimensional
state spaces. 
Simple methods of pursuit learning and reinforcement learning, which try to 
converge to the repeated selection of an optimal single index that provides a maximum mean reward (without a-priori knowledge
of the  average rewards for each index), 
are considered in \cite{pursuit-learning} and applied to wireless rate selection in \cite{pursuit-learning-wireless}.
Our stage-1 decision options can be viewed as a finite set of indices, and hence our problem is  
related to  \cite{pursuit-learning} \cite{pursuit-learning-wireless}.  However, our 2-stage problem 
structure and  the underlying stochastic queues, convex cost optimization, 
and multi-dimensional 
inequality constraints, make our problem much more complex.  Further, the optimal policy 
may (and typically does) result in a probabilistic mixture of many different action modes, rather than a single fixed action.

\section{The Max Weight Learning Problem} 

Consider a collection of $L$ discrete time queues $\bv{Q}(t) = (Q_1(t), \ldots, Q_L(t))$ with 
dynamic equation: 
\begin{equation} \label{eq:q-dynamics} 
Q_l(t+1) = \max[Q_l(t) - \mu_l(t), 0] + A_l(t)
\end{equation} 
where $A_l(t)$ is the amount of new arrivals to queue $l$ on slot $t$, and
$\mu_l(t)$ is the queue $l$ server rate on slot $t$.  These quantities 
are possibly affected by the two-stage control decision at slot $t$. 
Specifically, let $\script{K} \defequiv \{1, \ldots, K\}$ represent the set of stage-1 
decision options, and let $k(t)$ represent
the stage-1 decision made by the controller at time $t$, for $t \in \{0, 1, 2, \ldots\}$.   
Recall that the corresponding
random vector $\bv{\omega}(t)$ that is revealed is conditionally i.i.d. over all 
slots for which $k(t) = k$, with 
distribution function $F_k(\bv{\omega})$ given by (\ref{eq:F-def}).   
The $F_k(\bv{\omega})$ distributions are unknown to the controller. 
Let $\script{I}$ be the (possibly infinite) set of stage-2 control actions, and 
let $I(t) \in \script{I}$ denote the 
stage-2 control action at time $t$.  

The arrival and service vectors $\bv{A}(t) = (A_1(t), \ldots, A_L(t))$ 
and $\bv{\mu}(t) = (\mu_1(t), \ldots, \mu_L(t))$ are determined by $k(t)$, $\bv{\omega}(t)$, 
$I(t)$ according to (known)
functions $\hat{a}_l(k(t), \bv{\omega}(t), I(t))$ and $\hat{\mu}(k(t), \bv{\omega}(t), I(t))$:\footnote{The
analysis is the same if $\hat{a}_l(\cdot)$, $\hat{\mu}_l(\cdot)$, $\hat{x}(\cdot)$ outcomes are
random but i.i.d. given $k(t)$, $\bv{\omega}(t)$, $I(t)$, with known means $\overline{\hat{a}}_l(\cdot)$, 
$\overline{\hat{\mu}}_l(\cdot)$, $\overline{\hat{x}}(\cdot)$ that are used in the decision
making part of the algorithm.} 
\begin{eqnarray*}
A_l(t) &=& \hat{a}_l(k(t), \bv{\omega}(t), I(t)) \\
\mu_l(t) &=& \hat{\mu}_l(k(t), \bv{\omega}(t), I(t))
\end{eqnarray*}
Likewise, the penalty vector $\bv{x}(t) = (x_1(t), \ldots, x_M(t))$ is 
determined by the (known) penalty functions $x_m(t) = \hat{x} _m(k(t), \bv{\omega}(t),I(t))$ for 
each $m \in \{1, \ldots, M\}$. 
The penalties are (possibly negative) real numbers, and we 
assume that the penalty functions are bounded by finite 
constants $x_{m}^{min}$ and $x_m^{max}$ for all $m \in \{1, \ldots, M\}$,
so that: 
\[ x_m^{min} \leq x_m(t) \leq x_m^{max} \: \: \mbox{for  all $t$} \]
Likewise, the queue arrivals and service rates are bounded as follows: 
\begin{eqnarray*}
0 \leq A_l(t) \leq A_l^{max} \: \mbox{ for all $t$} \\
0 \leq \mu_l(t) \leq \mu_l^{max} \: \mbox{ for all $t$} 
\end{eqnarray*}
Aside from this boundedness, 
the  functions $\hat{a}_l(\cdot)$, $\hat{\mu}_l(\cdot)$, and 
$\hat{x}_m(\cdot)$ are otherwise arbitrary (possibly nonlinear, 
non-convex, and discontinuous). 
Define the time average penalty $\overline{\bv{x}}(t)$, averaged over the first $t$ slots, as follows: 
\begin{equation*} 
 \overline{\bv{x}}(t) \defequiv \frac{1}{t} \sum_{\tau=0}^{t-1} \expect{\bv{x}(\tau)}  
 \end{equation*} 

Let $f(\bv{x})$ be a
convex and continuous function over $\bv{x} \in \mathbb{R}^M$ (possibly negative, non-monotonic, and
non-differentiable).   Let $h_n(\bv{x})$ for $n \in \{1, \ldots, N\}$ be a collection of linear functions 
over  $\bv{x} \in \mathbb{R}^M$.  Note that since the $\bv{x}(t)$
penalties are bounded, the values of $f(\bv{x}(t))$ and $h_n(\bv{x}(t))$ are also bounded. 
The goal is to design a control policy that makes
2-stage decisions over time so as to solve the following problem:\footnote{While we assume the objective
function $f(\bv{x})$ is a general convex (possibly non-linear) function, for simplicity we assume
the cost functions $h_n(\bv{x})$ are linear (see Remark 1 in Section \ref{subsection:performance} 
for extensions to non-linear $h_n(\bv{x})$ functions).  
Example linear constraints for a wireless system 
are \emph{average power constraints} at each node, where $h_n(\bv{x})$ is a linear function that sums
the relevant components of the penalty vector $\bv{x}(t)$ that correspond to instantaneous power expenditure
at node $n$, and  $b_n$ represents the average power
constraint of node $n$.  
A  typical non-linear objective for networks is the maximization of
a concave utility function $g(\bv{x})$ of the time average throughput, 
where $g(\bv{x})$ selects
only those entries $x_m$ that correspond to throughput, and  $f(\bv{x}) = -g(\bv{x})$.} 
\begin{eqnarray}
\mbox{Minimize:} &  \limsup_{t\rightarrow\infty} f(\overline{\bv{x}}(t)) \label{eq:min1} \\
\mbox{Subject to:}&  \limsup_{t\rightarrow\infty} h_n(\overline{\bv{x}}(t)) \leq b_n \: \mbox{ for $n \in \{1, \ldots, N\}$} \label{eq:st1} \\
& \mbox{Stability of all queues $Q_1(t), \ldots, Q_L(t)$} \label{eq:stability-c1} 
\end{eqnarray}
In cases when the time average penalty vector converges to some value $\overline{\bv{x}}$, the $\limsup$ is 
equal to the regular limit and the above problem can be more simply stated as minimizing $f(\overline{\bv{x}})$
subject to $h_n(\overline{\bv{x}}) \leq b_n$ for all $n \in \{1, \ldots, N\}$ and to stability of all queues. 
The following notion of queue stability is used: 

\begin{defn} A discrete time queue is \emph{strongly stable} if: 
\[ \limsup_{t\rightarrow\infty}\frac{1}{t}\sum_{\tau=0}^{t-1} \expect{|Q(\tau)|} < \infty \]
\end{defn} 

We shall use the term \emph{stability} throughout to refer to strong stability.  The definition
above uses the absolute value of queue size because we shall soon introduce additional virtual 
queues that can take negative values.

\subsection{Auxiliary Variables for Nonlinear Cost Functions}

It is useful to write the cost function $f(\bv{x})$ 
as a sum of linear (or affine) and non-linear components. 
Specifically, define 
 $\tilde{\script{M}}$ as the set of all indices $m \in \{1, \ldots, M\}$ for which there are penalty 
variables $x_m(t)$ that participate in a \emph{non-linear component} of $f(\bv{x})$.  Then we can 
write $f(\bv{x})$  as follows: 
\begin{eqnarray*}
f(\bv{x}) = l(\bv{x}) + \tilde{f}(\tilde{\bv{x}}) 
\end{eqnarray*}
where $l(\bv{x})$ is a linear (or affine) function, $\tilde{\bv{x}} = (x_m)\left|\right._{m\in\tilde{\script{M}}}$ is a ``sub-vector'' of $\bv{x}$
that contains only entries $x_m$ for $m \in \tilde{\script{M}}$, and $\tilde{f}(\tilde{\bv{x}})$ 
are convex functions (and typically non-linear).   Such a decomposition is always possible, and 
in principle we can choose the trivial decomposition $\tilde{\script{M}} = \{1, \ldots, M\}$, 
$l(\bv{x}) = 0$, $\tilde{\bv{x}} = \bv{x}$, which does not attempt to exploit linearity even if it exists
in the cost function. 
However, it is useful to separate out the linear components,  
because we shall require one \emph{auxiliary variable} $\gamma_m(t)$ for each penalty $x_m(t)$ that 
 participates in a non-linear component of a cost function, while no such auxiliary 
 variable is required for 
 penalties that do not participate in any non-linear components.\footnote{While it is possible to always define 
 one auxiliary variable per penalty, 
 exploiting linearity and reducing the number of 
 auxiliary variables can be more direct and may lead to faster convergence times.}  

 For each $m \in \tilde{\script{M}}$, let $\gamma_m(t)$ be a new 
 variable
that can be chosen as desired on  each timeslot $t$, subject only to the constraint that:
\begin{equation} \label{eq:gamma-constraint} 
 x_m^{min} - \sigma \leq 
\gamma_m(t) \leq x_m^{max} + \sigma \: \mbox{ for all $m \in \tilde{\script{M}}$} 
\end{equation} 
for some positive value $\sigma>0$ (to be chosen later). 
Let $\bv{\gamma}(t) = (\bv{\gamma}_m(t))\left|\right._{m \in\tilde{\script{M}}}$ be a vector of $\gamma_m(t)$ components for $m \in \tilde{\script{M}}$.  Define the time average $\overline{\bv{\gamma}}(t)$ as follows: 
\[ \overline{\bv{\gamma}}(t) \defequiv \frac{1}{t}\sum_{\tau=0}^{t-1} \expect{\bv{\gamma}(\tau)} \]
Then it is not difficult to show 
that the problem (\ref{eq:min1})-(\ref{eq:stability-c1})
is equivalent to the following: 
\begin{eqnarray}
\mbox{Minimize:} & \limsup_{t\rightarrow\infty} \left[ l(\overline{\bv{x}}(t)) + \tilde{f}(\overline{\bv{\gamma}}(t)) \right] \label{eq:min2} \\
\mbox{Subject to:} & \limsup_{t\rightarrow\infty}  h_n(\overline{\bv{x}}(t))  \leq b_n \: \mbox{ for $n \in \{1, \ldots, N\}$} \label{eq:st2} \\
& \lim_{t\rightarrow\infty} \left[ \overline{x}_m(t) - \overline{\gamma}_m(t)\right] = 0 \: \mbox{ for $m \in \tilde{\script{M}}$}  \label{eq:st2-aux} \\
& \mbox{Stability of all queues $Q_1(t), \ldots, Q_L(t)$} \label{eq:stability-c2} 
\end{eqnarray}
Indeed, the equality constraint (\ref{eq:st2-aux}) indicates that the auxiliary variable $\gamma_m(t)$ can be used
as a proxy for $x_m(t)$ for all $m \in \tilde{\script{M}}$, so that the above problem is equivalent to (\ref{eq:min1})-(\ref{eq:stability-c1}). 
This is useful for stochastic optimization 
because $\gamma_m(t)$ can be chosen deterministically as any real number that satisfies (\ref{eq:gamma-constraint}),
whereas the penalty $x_m(t)$ has random outcomes.
These auxiliary variables are similar to those introduced in \cite{neely-fairness-infocom05} \cite{now}
for optimizing a convex and non-linear function of a time average penalty
in a stochastic network,  which is a more general (and more complex) problem than that of 
 optimizing a time average of a non-linear penalty function.   In the special case when the 
 objective function $f(\bv{x})$ is itself
 linear (so that $\tilde{f}(\bv{x}) = 0$ and $f(\bv{x}) = l(\bv{x})$), then 
 no auxiliary variables are needed, the set $\tilde{\script{M}}$ is empty, 
 and the constraints (\ref{eq:st2-aux}) are irrelevant. 
 
 \subsection{Virtual Queues for Time Average Inequalities and Equalities}

   To satisfy the time average 
 inequality constraints in (\ref{eq:st2}), we define one \emph{virtual queue} $U_n(t)$ for each 
 $n \in \{1, \ldots, N\}$, with dynamic queueing equation: 
 \begin{equation} \label{eq:u-dynamics} 
 U_n(t+1) = \max\left[U_n(t) + h_n(\bv{x}(t))   - b_n, 0\right] 
 \end{equation} 
 This can be viewed as a discrete time 
 queueing system with  a constant ``service rate'' $b_n$ and  with  
 arrivals $h_n(\bv{x}(t))$,  although we note in this case that the 
 ``arrivals'' and/or the ``service rate'' can potentially be negative on a given slot $t$.  
  The intuition is
 that stabilizing this virtual queue ensures that the time average ``arrival rate'' is less than or 
 equal to $b_n$.  This is similar to the virtual queues used for average power 
 constraints in \cite{neely-energy-it} 
 and average penalty constraints in \cite{now}.
 
  To satisfy the  time average equality constraints in (\ref{eq:st2-aux}), 
  we introduce a  \emph{generalized virtual queue} $Z_m(t)$ for each $m \in \tilde{\script{M}}$, with dynamic equation: 
 \begin{equation} \label{eq:z-dynamics} 
 Z_m(t+1) = Z_m(t) - \gamma_m(t) + x_m(t)
 \end{equation} 
 This  has a different structure because it enforces an equality constraint, and it can 
 be either positive or negative.  
 The following lemma shows that stabilizing the queues $U_n(t)$ and $Z_m(t)$ 
 ensures that the corresponding
 inequality and equality constraints are satisifed.

\begin{lem} \label{lem:q-stable} (Queue Stability Lemma) If the queues $U_n(t)$  and 
$Z_m(t)$ satisfy the following (for all $n \in \{1, \ldots, N\}$ and $m \in \tilde{\script{M}}$): 
\begin{equation} \label{eq:rate-stable} 
\lim_{t\rightarrow\infty} \frac{\expect{U_n(t)}}{t} = 0 \: \: , \: \: \lim_{t\rightarrow\infty} \frac{\expect{|Z_m(t)|}}{t} = 0
\end{equation} 
Then all inequality constraints (\ref{eq:st2}) and  (\ref{eq:st2-aux}) are satisfied. 
Further, the condition (\ref{eq:rate-stable}) holds whenever the queues are strongly stable. 
\end{lem} 
\begin{proof}
Omitted for brevity (see \cite{neely-energy-it} for a related proof). 
\end{proof}

\subsection{Lyapunov Functions} 

Define $\bv{\Theta}(t) \defequiv [\bv{Q}(t); \bv{U}(t); \bv{Z}(t)]$ as the 
vector of all actual and virtual queue backlogs. 
To stabilize the queues, we define the following Lyapunov function: 
\[ L(\bv{\Theta}(t)) \defequiv \frac{1}{2}\sum_{l=1}^L Q_l(t)^2 + \frac{1}{2}\sum_{n=1}^N U_n(t)^2 +  \frac{1}{2}\sum_{m\in\tilde{\script{M}}} Z_m(t)^2 \]
Note that this Lyapunov function grows large when the absolute vale of queue size is large, and hence keeping
this function small also maintains stable queues. Define the \emph{one-step conditional Lyapunov drift} as follows:\footnote{Strictly 
speaking, notation should be $\Delta(\bv{\Theta}(t), t)$, as the drift may be non-stationary.  
However, we use the simpler notation $\Delta(\bv{\Theta}(t))$ as a formal 
representation of the right hand side of (\ref{eq:delta-drift}).} 
\begin{equation} \label{eq:delta-drift}  
\Delta(\bv{\Theta}(t)) \defequiv \expect{L(\bv{\Theta}(t+1)) - L(\bv{\Theta}(t))\left|\right.\bv{\Theta}(t)} 
\end{equation} 

Let $V$ be a non-negative parameter used to control the proximity of our algorithm to the optimal solution 
of (\ref{eq:min2})-(\ref{eq:stability-c2}).  Using the framework of \cite{now},  
we consider a  control policy that observes
the queue backlogs $\bv{\Theta}(t)$ and takes control actions on each slot $t$ 
that minimize a bound on the following ``drift plus penalty'' expression: 
\[   \Delta(\bv{\Theta}(t)) + \expect{V  l(\bv{x}(t))  + V\tilde{f}(\bv{\gamma}(t)) \left|\right.\bv{\Theta}(t)} \] 
Computing the Lyapunov drift $\Delta(\bv{\Theta}(t))$ by squaring the queueing update equations
(\ref{eq:u-dynamics}), (\ref{eq:z-dynamics}),  (\ref{eq:q-dynamics}) and taking conditional expectations leads to the 
following lemma.

\begin{lem} \label{lem:main-drift} (The $RHS(\cdot)$ Bound) For a general control policy we have: 
\begin{eqnarray}
\Delta(\bv{\Theta}(t)) + \expect{V l(\bv{x}(t))  + V\tilde{f}(\bv{\gamma}(t)) \left|\right.\bv{\Theta}(t)}  \leq B   \nonumber \\
+ \expect{V  l(\bv{x}(t))  + V\tilde{f}(\bv{\gamma}(t)) \left|\right.\bv{\Theta}(t)} \nonumber \\
- \sum_{n=1}^N U_n(t)\expect{ b_n  - h_n(\bv{x}(t))  \left|\right.\bv{\Theta}(t)}  \nonumber \\
- \sum_{m \in \tilde{\script{M}}} Z_m(t)\expect{ \gamma_m(t) - x_m(t) \left|\right.\bv{\Theta}(t)} \nonumber \\
- \sum_{l=1}^L Q_l(t)\expect{  \mu_l(t) - A_l(t) \left|\right.\bv{\Theta}(t)} \label{eq:main-drift} 
\end{eqnarray}
where $B$ is a  finite constant that satisfies the following for all $t$ and all possible control actions that can 
be taken on slot $t$:
\begin{eqnarray*}
B  &\geq& \sum_{n=1}^N \expect{(b_n  - h_n(\bv{x}(t)))^2 \left|\right.\bv{\Theta}(t)} \\
&&  + \sum_{m\in\tilde{M}} \expect{(\gamma_m(t) - x_m(t))^2\left|\right.\bv{\Theta}(t)} \\
&& + \sum_{l=1}^L \expect{(\mu_l(t) - A_l(t))^2\left|\right.\bv{\Theta}(t)} 
\end{eqnarray*}
Such a constant $B$ exists because of the boundedness assumptions
of the penalty and cost functions, and an explicit bound  can be determined
by considering the maximum squared values attained by the penalties and costs. 
\end{lem} 
\begin{proof} 
The proof is a straightforward drift computation (see, for example, \cite{now}), and is omitted for brevity.
\end{proof}

The next section analyzes the performance of policies that choose control actions every slot to 
(approximately) minimize the right hand side of the drift expression (\ref{eq:main-drift}). 

\subsection{The Performance Theorem} \label{subsection:performance} 

Define  $f^*$ as the optimal solution for the problem  (\ref{eq:min1})-(\ref{eq:stability-c1}) (i.e., it
is the infimum cost over all policies that satisfy the constraints). 
Define a value $\theta$ such that $0 \leq \theta < 1$, and 
consider the class of restricted policies that have random \emph{exploration
events} independently with probability $\theta$ every slot.  If a given slot
$t$ is an exploration event, the stage-1 decision $k(t)$ is 
chosen independently and uniformly over $\{1, \ldots, K\}$ (regardless of the 
state of the system at this time). We say that
the slot is an \emph{exploration event of type $k$} if the exploration event
leads to the random choice of option $k$.   Hence, exploration events of type
$k$ occur independently with probability $\theta/K$ every slot.   We note
that the stage-2 decision $I(t)$ and the auxiliary variables $\bv{\gamma}(t)$ 
can be chosen arbitrarily on every slot, regardless of whether or not the slot
is an exploration event. 

If $\theta >0$, the exploration events ensure that each stage-1 control option is tested
infinitely often.
Define $f^*_{\theta}$ as the optimal solution of (\ref{eq:min1})-(\ref{eq:stability-c1}) subject 
to the additional
constraint that such random exploration events are imposed. 
  It shall be convenient to define optimality in terms of $f^*_{\theta}$. 
It is clear that $f^*_{0} = f^*$, and intuitively one expects 
that $f^*_{\theta} \rightarrow f^*$ as $\theta \rightarrow 0$.\footnote{Specifically, it can be
shown that $f^*_{\theta} \rightarrow f^*$ whenever  $\epsilon_{max}>0$, where $\epsilon_{max}$ is 
defined in Assumption 2.}   Further, in systems where the optimal $f^*$ can be achieved by a policy
that chooses  
each stage-1 control option a positive fraction of time, it can be shown that there exists a positive
value $\theta^*$ such that $f^* = f^*_{\theta}$ 
whenever $0 \leq \theta \leq \theta^*$. 
We now assume the following properties hold concerning stationary and randomized control 
policies with random exploration events of probability $\theta$. 

\emph{Assumption 1 (Feasibility):} There is a stationary and randomized policy that chooses a stage-1 control action
$k^*(t) \in \script{K}$ according to a fixed probability distribution such that each option is chosen 
with probability at least $\theta/K$ (revealing a corresponding 
random vector $\bv{\omega}^*(t)$), and chooses a stage-2 control action $I^*(t) \in \script{I}$ 
as a potentially randomized function of $\bv{\omega}^*(t)$,  such that: 
\begin{eqnarray}
&l(\expect{\bv{x}^*(t)}) + \tilde{f}(\bv{\gamma}^*) = f^*_{\theta}&  \label{eq:ass1-1}\\
& b_n - h_n(\expect{\bv{x}^*(t)}) \geq 0  \: \mbox{ for all $n \in \{1, \ldots, N\}$}& \label{eq:ass1-2}\\
&\expect{\mu_l^*(t)} - \expect{A_l^*(t)} \geq 0 \: \mbox{ for all $l \in \{1, \ldots, L\}$}&  \label{eq:ass1-3}
\end{eqnarray}
where 
 $\bv{x}^*(t)$, $\bv{\mu}^*(t)$, $\bv{A}^*(t)$ are the penalty, service rate, and arrival
vectors corresponding to the stationary and randomized policy, defined by: 
\begin{eqnarray*}
\bv{x}^*(t) &=& \hat{x}(k^*(t), \bv{\omega}^*(t), I^*(t)) \\
\bv{\mu}^*(t) &=& \hat{\mu}(k^*(t), \bv{\omega}^*(t), I^*(t)) \\
\bv{A}^*(t) &=& \hat{a}(k^*(t), \bv{\omega}^*(t), I^*(t)) 
\end{eqnarray*}
and where $\bv{\gamma}^*$ is a vector with components $(\gamma_m^*)|_{m\in\tilde{\script{M}}}$
such that $\gamma_m^* \defequiv \expect{x_m^*(t)}$ for all $m \in \tilde{\script{M}}$.  Note
that $x_m^{min} \leq x_m(t) \leq x_m^{max}$ always, and so  
$x_m^{min} \leq \gamma_m^* \leq x_m^{max}$ for all $m \in \tilde{\script{M}}$.  Thus, each component 
$\gamma_m^*$  satisfies the required auxiliary variable constraint (\ref{eq:gamma-constraint}). 

This assumption states that the problem is feasible, and that the optimal $f^*_{\theta}$ value 
can be achieved by a particular stationary and randomized policy that meets the 
time average penalty constraints and ensures the time average service rate is greater than
or equal to the time average arrival rate in all queues.\footnote{See \cite{neely-energy-it} for a 
proof that optimality can be defined over the class of stationary, randomized policies for
minimum power problems.}   The next assumption states that the constraints are not only feasible, 
but have a useful slackness property.

\emph{Assumption 2 (Slackness of Constraints):}  There is a  
value $\epsilon_{max}>0$ together with a
stationary and randomized policy that makes stage-1 and stage-2
control decisions $k'(t) \in \script{K}$ and $I'(t) \in \script{I}$ such that each stage-1 option
is chosen with probability at least $\theta/K$, and: 
\begin{eqnarray}
&  b_n - h_n(\expect{\bv{x}'(t)}) \geq  \epsilon_{max} \: \mbox{ for all $n \in \{1, \ldots, N\}$}& \label{eq:prime1}\\
&\expect{\mu_l'(t)} - \expect{A_l'(t)} \geq \epsilon_{max} \: \mbox{ for all $l \in \{1, \ldots, L\}$}&  \label{eq:prime2}
\end{eqnarray} 
where $\bv{x}'(t)$, $\bv{\mu}'(t)$, $\bv{A}'(t)$ are the penalty, service rate, and arrival
vectors corresponding to the decisions $k'(t)$ and $I'(t)$.

Now define $RHS(t, \bv{\Theta}(t), k(t), I(t), \bv{\gamma}(t))$ as the right 
hand side of the drift bound (\ref{eq:main-drift})
with a given queue state $\bv{\Theta}(t)$ and control actions $k(t)$, $I(t)$, $\bv{\gamma}(t)$ 
at time $t$.   Given a particular queue state $\bv{\Theta}(t)$, 
define the \emph{max-weight} control decisions $k^{mw}(t)$, $I^{mw}(t)$, 
$\bv{\gamma}^{mw}(t)$
as the ones that minimize the following conditional expectation over all alternative feasible control
actions that can be made on slot $t$ (subject to the $\theta$ exploration probability):\footnote{For simplicity, we 
implicitly assume that the infimum of (\ref{eq:mw-def}) 
over all feasible control actions is achieved by a particular set of decisions, called the max-weight decisions.  
Else, the results can be recovered by defining the max-weight decisions according to a sequence of policies that 
converge to the infimum.} 
\begin{equation} \label{eq:mw-def} 
 \expect{RHS(t, \bv{\Theta}(t), k(t), I(t), \bv{\gamma}(t)) \left|\right.\bv{\Theta}(t)}  
 \end{equation} 
 Note that the $k^{mw}(t)$ decisions are still determined randomly in the case of exploration events of probability
 $\theta$, but  are chosen 
 to maximize the above expression whenever the current slot does not have an exploration event.

The auxiliary vector 
$\bv{\gamma}(t)$ appears in separable terms on the right hand side of (\ref{eq:main-drift}), 
and so the policy $\bv{\gamma}^{mw}(t)$ can be determined separately from the $k^{mw}(t)$ and 
$I^{mw}(t)$ decisions.  It is computed by first observing the queue backlogs  $Z_m(t)$
on each slot $t$, and choosing $\bv{\gamma}^{mw}(t)$ as the solution 
to the following deterministic convex optimization: 
\begin{eqnarray}
\mbox{Minimize:} &   V\tilde{f}(\bv{\gamma}(t))  - \sum_{m\in\tilde{\script{M}}} Z_m(t)\gamma_m(t)   \label{eq:gamma-opt-mw} \\
\mbox{Subject to:} & x_{m}^{min} - \sigma \leq \gamma_m(t) \leq x_m^{max} + \sigma \: \mbox{ for all $m \in \tilde{\script{M}}$} \label{eq:gamma-constraint-mw} 
\end{eqnarray}
If the non-linear function $\tilde{f}(\bv{\gamma})$ is 
separable in the $\bv{\gamma}$ vector (as is the case in many network optimization problems), the
above optimization amounts to separately finding 
$\gamma_m^{mw}(t)$ (for each $m \in \tilde{\script{M}}$) 
as the minimum of a convex single-variable function over 
the closed interval defined by (\ref{eq:gamma-constraint-mw}).   

While the $\bv{\gamma}^{mw}(t)$ can thus be computed, it is more challenging to 
determine the  stage-1 and stage-2 decisions that minimize the right hand 
side of (\ref{eq:main-drift}), as this would require 
knowledge of the probability distributions $F_k(\bv{\omega})$.  We thus seek 
an \emph{approximation} to the $k^{mw}(t)$ and $I^{mw}(t)$ policies. 
Suppose
the following additional assumption holds concerning such an approximation. 

\emph{Assumption 3 (Approximate Scheduling):} Every slot $t$ the queue backlogs $\bv{\Theta}(t)$ 
are observed and 
control decisions $k(t) \in \script{K}$ (subject to exploration events with probability $\theta$), 
 $I(t) \in \script{I}$, and 
$\bv{\gamma}(t)$ satisfying (\ref{eq:gamma-constraint}) are made to ensure the following: 
\begin{eqnarray}
\expect{RHS(t, \bv{\Theta}(t), k(t), I(t), \bv{\gamma}(t))}  &\leq& \expect{RHS(t, \bv{\Theta}(t), k^{mw}(t), I^{mw}(t), \bv{\gamma}^{mw}(t))} \nonumber \\
&& + C  + V\epsilon_V \nonumber \\
&& + \sum_{n=1}^N \expect{U_n(t)}\epsilon_U   + \sum_{m\in\tilde{\script{M}}} \expect{|Z_m(t)|}\epsilon_Z  + \sum_{l=1}^L \expect{Q_l(t)} \epsilon_Q \label{eq:approx-assumption} 
\end{eqnarray}
where $C$, $\epsilon_V$,  $\epsilon_U$, $\epsilon_Z$, $\epsilon_Q$ are non-negative 
constants (independent of 
$t$). The 
expectation on the left hand side is with respect to the current queue state $\bv{\Theta}(t)$ and the 
actual decisions $k(t)$, $I(t)$, $\bv{\gamma}(t)$ implemented, 
while the expectation on the right is with respect to the current queue state $\bv{\Theta}(t)$
and the (possibly not implemented) max-weight decisions $k^{mw}(t)$, $I^{mw}(t)$, $\bv{\gamma}^{mw}(t)$ 
that minimize the right hand side of (\ref{eq:main-drift}).

We note that the structure of the approximation bound in (\ref{eq:approx-assumption}) is typical 
for algorithms that attempt to select a control action based on imperfect knowledge of the probability 
distributions
of the resulting $\bv{x}(t)$, $\bv{\mu}(t)$, $\bv{A}(t)$ vectors, as the resulting approximations are typically
proportional to the $V$ constant and the $U_n(t)$, $|Z_m(t)|$, and $Q_l(t)$ queue sizes on the right hand side of
(\ref{eq:main-drift}).  In the case of perfect implementation of the max-weight policy $k^{mw}(t)$, 
$I^{mw}(t)$, $\bv{\gamma}^{mw}(t)$, we 
have $\epsilon_V = \epsilon_U = \epsilon_Z = \epsilon_Q   = 0$ and $C=0$.

\begin{thm} \label{thm:1} (Performance Theorem) Suppose Assumptions 1 and 2 hold, and 
that a control algorithm is implemented that satisfies Assumption 3 with fixed control 
parameters $V \geq 0$
and $\sigma > 0$. 
Suppose $\epsilon_Q$, $\epsilon_Z$, $\epsilon_U$ are small enough and $\sigma$  is chosen large
enough to satisfy the following:
\begin{eqnarray} 
\epsilon_U < \epsilon_{max}  \: \: , \epsilon_Z < \sigma  \: \: , \: \: \epsilon_Q < \epsilon_{max} 
\end{eqnarray} 
Then all time average constraints (\ref{eq:st2})-(\ref{eq:stability-c2}) hold.  In particular, 
all queues are strongly stable and satisfy  for all $t$: 
\begin{eqnarray}
\frac{1}{t}\sum_{\tau=0}^{t-1}\left[\sum_{n=1}^N\expect{U_n(\tau)} + \sum_{m\in\tilde{\script{M}}} \expect{|Z_m(\tau)|}
+ \sum_{l=1}^L \expect{Q_l(\tau)} 
\right] \leq \nonumber \\
\frac{B + C + V(l_{diff} + \tilde{f}_{diff} + \epsilon_V)}{\epsilon_{approx}} + 
\frac{\expect{L(\bv{\Theta}(0))}}{\epsilon_{approx}t} \label{eq:thm1-q} 
\end{eqnarray}
where $\epsilon_{approx}$ is defined: 
\[ \epsilon_{approx} \defequiv \min[\epsilon_{max} - \epsilon_U, \epsilon_{max} - \epsilon_Q,  \sigma - \epsilon_Z] \]
and where $l_{diff}$ and $\tilde{f}_{diff}$ are finite bounds that satisfy: 
\begin{eqnarray*}
l(\bv{x}_1) - l(\bv{x}_2) \leq l_{diff} &  \mbox{for any $\bv{x}_1$, $\bv{x}_2$ in the set $\bv{x}^{min} \leq \bv{x} \leq \bv{x}^{max}$} \\
\tilde{f}(\bv{\gamma}_1) - \tilde{f}(\bv{\gamma}_2) \leq \tilde{f}_{diff} & \mbox{for any $\bv{\gamma}_1, \bv{\gamma}_2$ in the set $\bv{x}^{min} - \bv{\sigma} \leq \bv{\gamma} \leq \bv{x}^{max} + \bv{\sigma}$}
\end{eqnarray*}
 Further, 
the time average cost satisfies:\footnote{The expression (\ref{eq:thm1-cost}) holds for all $t$ (without the $\limsup$)
in the special case when $\bv{\Theta}(0) = \bv{0}$ and 
$f(\bv{x})$ is linear so that $f(\bv{x}) = l(\bv{x})$.  The rate at which the limit 
converges in the general (non-linear) case is proportional to the rate at which the time average
expectations of $\gamma_m(t)$ converge to the time average expectations of $x_m(t)$ for each
$m \in \tilde{\script{M}}$, which is roughly the average of  $|Z_m(t)|/t$. This is highlighted in the proof
of the theorem, see inequality (\ref{eq:utility-proof}).} 
\begin{eqnarray}
\limsup_{t\rightarrow\infty} f(\overline{\bv{x}}(t)) \leq f^*_{\theta} + \epsilon_V + \delta + (B + C)/V \label{eq:thm1-cost} 
\end{eqnarray}
where we recall that $f(\bv{x}) = l(\bv{x}) + \tilde{f}(\tilde{\bv{x}})$ and $f^*_{\theta}$ is the optimal solution 
of (\ref{eq:min2})-(\ref{eq:stability-c2}) subject to exploration events with probability $\theta$, and 
where $\delta$ is defined: 
\[  \delta \defequiv (l_{diff} + \tilde{f}_{diff})\max\left[  \frac{\epsilon_U}{\epsilon_{max}}, \frac{\epsilon_Z}{\sigma}, \frac{\epsilon_Q}{\epsilon_{max}} \right]   \]
\end{thm}

Theorem \ref{thm:1} states that, under the given approximation assumptions, the  algorithm
stabilizes all queues and 
yields a time average cost that is within $\epsilon_V +\delta + O(1/V)$ of the optimal value $f^*_{\theta}$. 
Hence, this bound can 
be made arbitrarily close to $f^*_{\theta} + \epsilon_V + \delta$ by  choosing $V$ suitably large, at the cost of 
a linear increase in average queue congestion with $V$.  Further, we note that the terms $\epsilon_V$ 
and $\delta$ tend to zero  as the error values $\epsilon_V$, $\epsilon_U$, 
$\epsilon_Z$, $\epsilon_Q$ tend to zero. 
 In the special case when the exact max-weight
policy is implemented every slot (so that every slot $t$ the controller makes 
decisions $k^{mw}(t)$, $I^{mw}(t)$, $\bv{\mu}^{mw}(t)$ 
that minimize the right hand side of (\ref{eq:main-drift})), then 
we have $C = 0$ and $\epsilon_V = \epsilon_U = \epsilon_Z = \epsilon_Q = \delta = 0$.   In this case, 
we can also choose $\theta = 0$ so that 
performance is within $O(1/V)$ of the optimal value $f^*$.  This special case  
is similar to the 
stochastic network optimization 
result of \cite{now}, with the exception that \cite{now} assumes  the convex cost function
$f(\bv{x})$ is non-decreasing in each entry of $\bv{x}$ (using auxiliary variables with 
``one-sided'' virtual queues that are always non-negative), whereas here we treat
a possibly non-monotonic cost function via  (possibly negative) virtual queues $Z_m(t)$. 

\begin{proof} (Theorem \ref{thm:1}) 
See Appendix A. 
\end{proof} 

The following related theorem uses a variable $V(t)$ parameter and allows for the 
uncertainty to tend to zero while achieving the exact penalty $f_{\theta}^*$. Its proof follows
as a simple consequence of the proof of Theorem \ref{thm:1}. 

\begin{thm} \label{thm:variable-v}  (Variable $V(t)$ parameter) 
Suppose Assumptions 1 and 2 hold.  Let $\beta_1$ and $\beta_2$ be values such that 
$0 < \beta_1 < \beta_2 < 1$.  Assume that after some finite time $t_0$,  we use a $V(t)$
parameter that increases with time, so that 
 $V(t) = (t - t_0 + 1)^{\beta_2}V_0$ for all $t \geq t_0$ and for some  constant $V_0>0$.  
 Assume the queue states at time $t_0$ are arbitrary but finite, and assume we make control
 decisions $k(t)$, $I(t)$, $\bv{\gamma}(t)$ such that the following
holds for all $t \geq t_0$ (which is a modification of Assumption 3): 
\begin{eqnarray*}
\expect{RHS(t, \bv{\Theta}(t), k(t), I(t), \bv{\gamma}(t))} &\leq& \expect{RHS(t, \bv{\Theta}(t), k^{mw}(t), I^{mw}(t), \bv{\gamma}^{mw}(t))} \\
&& + C(t)  + V(t)\epsilon_V(t) \\
&& + \sum_{n=1}^N \expect{U_n(t)}\epsilon_U(t) + \sum_{m \in \tilde{\script{M}}  }\expect{ |Z_m(t)|}\epsilon_Z(t) +\sum_{l=1}^L \expect{Q_l(t)}\epsilon_Q(t)
\end{eqnarray*}
where $C(t)$, $\epsilon_V(t)$, $\epsilon_U(t)$, $\epsilon_Z(t)$, $\epsilon_Q(t)$ are deterministic functions of time such that: 
\begin{eqnarray*}
\lim_{t\rightarrow\infty} \epsilon_x(t) = 0 
\end{eqnarray*}
where $x \in \{V, U, Z, Q\}$, and where: 
\[ C(t) \leq O((t - t_0 + 1)^{\beta_1}) \: \: \mbox{for $t \geq t_0$}  \]
Then the time average constraints (\ref{eq:st2})-(\ref{eq:st2-aux}) hold, and all queues $Q_l(t)$ are \emph{mean rate stable}, in the 
sense that: 
\[ \lim_{t\rightarrow\infty} \frac{\expect{Q_l(t)}}{t} = 0 \: \: \mbox{ for all $l \in \{1, \ldots, L\}$} \]
Further, the time average cost converges to the optimal value $f_{\theta}^*$: 
\[ \lim_{t\rightarrow\infty} f(\overline{\bv{x}}(t)) = f_{\theta}^* \]
\end{thm}
\begin{proof} 
See Appendix B. 
\end{proof} 

This method of using an increasing $V(t)$ parameter  can be viewed as a stochastic analogue
of classic diminishing step-size 
methods for  static optimization problems \cite{bertsekas-convex}.  We note that 
$C(t)$  is assumed to increase at a rate slower than that of $V(t)$, while the $\epsilon_x(t)$ functions can converge
to zero with any rate.  Note that mean rate stability 
is a weak form of stability, and does \emph{not} imply that average queue sizes and delays are finite.  In fact, typically 
average congestion and delay 
are \emph{necessarily} infinite when exact cost optimization is achieved \cite{berry-fading-delay} \cite{neely-energy-delay-it}. 

\emph{Remark 1:} The results of Theorems \ref{thm:1} and \ref{thm:variable-v} can be 
generalized to allow the $h_n(\bv{x})$ functions to be 
convex (possibly non-linear) by using one auxiliary variable $\gamma_m(t)$ for each 
penalty $x_m(t)$, in which case the constraints (\ref{eq:st2-aux}) can be enforced
by modifying the virtual queues $U_n(t)$ in (\ref{eq:u-dynamics}) to $\hat{U}_n(t)$ with dynamics:
\[ \hat{U}_n(t+1) = \max[\hat{U}_n(t) + h_n(\bv{\gamma}(t)) - b_n, 0] \]
This has the disadvantage of creating more virtual queues (one for each penalty $m \in \script{M}$
rather than one for each penalty $m \in \tilde{\script{M}}$), but 
has the advantage of allowing for non-linear $h_n(\bv{x})$ functions.  It has the additional advantage
of removing the uncertain $\bv{x}(t)$ penalties from the
drift terms corresponding to the queues $\hat{U}_n(t)$. This ensures
$\epsilon_U = 0$ whenever the auxiliary variables are chosen 
according to the max-weight rule $\bv{\gamma}^{mw}(t)$ (which, due to separability, 
does not require knowledge of the $F_k(\bv{\omega})$
distributions).  Similarly, one can also use auxiliary variables in the cost function 
$f(\bv{\gamma}(t))$ (as a proxy for the $f(\bv{x}(t))$ values),  so that $\epsilon_V = 0$. 
With these modifications,  all uncertainty is isolated to $\epsilon_Z$ and $\epsilon_Q$.

\emph{Remark 2:} Theorems \ref{thm:1} and \ref{thm:variable-v} can be used for any form of approximate scheduling, 
including cases when the optimal $I(t)$ decision involves a complex combinatorial choice
that can only be approximated (or when the optimization for the auxiliary variable $\bv{\gamma}(t)$ 
is approximate).  
This is related to similar approximate scheduling results developed for systems 
without stage-1 decisions in 
\cite{now} \cite{neely-divbar-journal} \cite{atilla-approximate-schedule} \cite{neely-thesis}. 
However, our main interest is when the approximation 
is due to the uncertainty in the probability distributions $F_k(\bv{\omega})$, and max-weight
learning algorithms for this context are developed in the next section. 

\section{Estimating the Max-Weight Functional}

Theorem \ref{thm:1} suggests that our control policy should make decisions 
for $k(t)$, $I(t)$, $\bv{\gamma}(t)$ every slot  in an effort to minimize the 
right hand side of (\ref{eq:main-drift}). The optimal auxiliary variable decisions 
$\bv{\gamma}^{mw}(t)$ 
for this goal
have already been established and are given by the solution of  (\ref{eq:gamma-opt-mw})-(\ref{eq:gamma-constraint-mw}). 
Note that these decisions do not require knowledge of the $F_k(\bv{\omega})$ distribution. Likewise, 
the optimal $I^{mw}(t)$ decision does not require knowledge of the $F_k(\bv{\omega})$ distribution. 
Specifically, given a collection of observed queue backlogs $\bv{\Theta}(t)$ and an
observed outcome $\bv{\omega}(t)$ (which is the result of the stage-1 decision $k(t)$
that is chosen), $I^{mw}(t)$ is defined as the optimal solution to the following (breaking ties 
arbitrarily): 
\begin{eqnarray}
\mbox{Minimize:} &  Vl(\hat{x}(k(t),  \bv{\omega}(t), I(t))) + \sum_{n=1}^N U_n(t)h_n(\hat{x}(k(t), \bv{\omega}(t), I(t))) + \nonumber \\
&  \sum_{m\in\tilde{\script{M}}} Z_m(t)\hat{x}_m(k(t), \bv{\omega}(t), I(t)) 
 - \sum_{l=1}^L Q_l(t)[\hat{\mu}_l(k(t), \bv{\omega}(t), I(t)) - \hat{a}_l(k(t), \bv{\omega}(t), I(t))] \label{eq:I-mw} \\
\mbox{Subject to:} &  I(t) \in \script{I} \nonumber
\end{eqnarray}
The complexity of making these $I^{mw}(t)$ decisions depends on the physical structure of the 
network.  The decisions are often trivial when the set $\script{I}$ contains only a finite (and small)
number of control options (such as when the decisions are to remain idle or serve a single queue),
in which case the function (\ref{eq:I-mw}) is simply compared on each of the different choices in 
$\script{I}$. 
For multi-hop networks with combinatorial resource allocation constraints, the choice of $I^{mw}(t)$
might be difficult, although constant-factor approximations are often possible (see 
\cite{now} \cite{neely-divbar-journal} \cite{atilla-approximate-schedule} \cite{neely-thesis}).

The optimal $k^{mw}(t)$ decisions can be defined in terms of the $I^{mw}(t)$ decisions as
follows:  On each slot $t$, $k^{mw}(t)$ is chosen as $k$, according an independent type-$k$ exploration event,
with probability $\theta/K$.  If no exploration event occurs on slot $t$ (which happens with probability 
$1-\theta$), the queue backlogs $\bv{\Theta}(t)$ are observed and 
$k^{mw}(t)$ is chosen as the value $k \in \{1, \ldots, K\}$ with the lowest value of 
$e_k(t)$ (breaking ties arbitrarily), where $e_k(t)$ is defined: 
\begin{eqnarray}
 e_k(t) \defequiv \expect{\min_{I \in \script{I}}\left[Y_k(I, \bv{\omega}(t), \bv{\Theta}(t))\right] \left|\right.k(t) = k, \bv{\Theta}(t)}  \label{eq:ek} 
 \end{eqnarray}
where $\bv{\omega}(t)$ is the random outcome that results from the stage-1 choice $k(t) = k$, and 
the 
function $Y_k(I, \bv{\omega}, \bv{\Theta})$ is  defined for a particular stage-2 decision $I$, outcome
$\bv{\omega}$, and queue state $\bv{\Theta} = [\bv{Q}; \bv{U}; \bv{Z}]$,  as follows:
\begin{eqnarray}
Y_k(I, \bv{\omega}, \bv{\Theta}) &\defequiv& 
Vl(\hat{x}(k, \bv{\omega}, I))    + \sum_{n=1}^{N} U_nh_n(\hat{x}(k, \bv{\omega}, I)) \nonumber \\
  && + \sum_{m\in\tilde{\script{M}}} Z_m\hat{x}_m(k, \bv{\omega}, I) \nonumber \\
 && - \sum_{l = 1}^{L} Q_l[\hat{\mu}_l(k, \bv{\omega}, I) - \hat{a}_l(k, \bv{\omega}, I)] \label{eq:Y}
\end{eqnarray}
Thus, $e_k(t)$ is the expected value of the expression (\ref{eq:I-mw}) over the distribution $F_k(\bv{\omega})$
for the $\bv{\omega}(t)$ random variable that arises from choosing $k(t) = k$, assuming that the
optimal stage-2 decision $I^{mw}(t)$ is then made.  However, computation
of the exact $e_k(t)$ values would typically require full knowledge of the probability distributions 
$F_k(\bv{\omega})$ (and the computation may be difficult even if these distributions are fully known). 
Rather than using the exact conditional expectations, we consider two forms of estimates.

\subsection{Estimating the $e_k(t)$ value --- Approach 1} 

Define an integer $W$ that represents a  \emph{moving average window size}.  For each stage-1 option 
$k \in \{1, \ldots, K\}$ and each time $t$, define $\bv{\omega}_1^{(k)}(t), \ldots, \bv{\omega}_W^{(k)}(t)$ 
as the actual $\bv{\omega}(\tau)$ outcomes observed over the last $W$ type-$k$ exploration events 
that took place before time $t$.  Define the estimate $\hat{e}_k(t)$ as follows: 
\[ \hat{e}_k(t) \defequiv \frac{1}{W}\sum_{w=1}^{W} \min_{I\in\script{I}}\left[Y_k(I, \bv{\omega}_{w}^{(k)}(t), \bv{\Theta}(t))\right] \]
In the case when there have not yet been $W$ previous type-$k$ exploration events by time $t$, 
the estimate $\hat{e}_k(t)$ is taken with respect to the (fewer than $W$) events, and is set to zero if no
such events have occurred.    The estimates $\hat{e}_k(t)$ can be viewed as empirical averages
of the function (\ref{eq:Y}), using the current queue backlogs $\bv{\Theta}(t) = [\bv{Q}(t); \bv{U}(t); \bv{Z}(t)]$
but using the outcomes $\bv{\omega}_w^{(k)}(t)$ observed on previous type-$k$ exploration events and the 
corresponding optimal stage-2 decisions. 

Note that one might define $\hat{e}_k(t)$ according to an average
over the past $W$ slots on which stage-1 decision $k$ has been made, rather than over the past
$W$ type-$k$ exploration events.  The reason we have used exploration events is to overcome the subtle
``inspection paradox'' issues involved in sampling the previous $\bv{\omega}(\tau)$ outcomes.  Indeed, 
even though $\bv{\omega}(\tau)$ is generated in an i.i.d. way every slot in which $k(\tau) = k$ is chosen, 
the distribution of the last-seen outcome $\bv{\omega}$ that corresponds to a particular decision $k$ may be
\emph{skewed} in favor of creating larger penalties.  This is because our algorithm may choose to avoid decision $k$
for a longer period of time if this last outcome was non-favorable.    Sampling at random type-$k$ exploration
events ensures that our samples indeed form an i.i.d. sequence.
An additional difficulty remains:  Even though these samples $\{\bv{\omega}_w^{(k)}(t)\}$ form an  i.i.d. sequence, 
they are \emph{not} independent of the queue values $\bv{\Theta}(t)$, as these prior outcomes have influenced the
current queue states. We overcome this difficulty in Section \ref{section:learning-analysis} via a delayed-queue
analysis.  

This form of estimation does not require knowledge of the $F_k(\bv{\omega})$ distributions.  However, evaluation 
of $\hat{e}_k(t)$   requires $W$ computations of the type (\ref{eq:I-mw}) on each slot $t$, according to the
value of each  particular $\bv{\omega}_w^{(k)}(t)$ vector.  This can be 
difficult in the case when $W$ is large, and hence the next subsection describes a second 
estimation approach that uses only one such computation per slot.

\subsection{Estimating the $e_k(t)$ value --- Approach 2} 

Again let $W$ be an integer 
moving average window size.  For each stage-1 decision $k \in \{1, \ldots, K\}$, 
 define 
$\bv{\omega}_1^{(k)}(t), \ldots, \bv{\omega}_W^{(k)}(t)$ the same as in Approach 1.  Further define
$\bv{\Theta}_1^{(k)}(t), \ldots, \bv{\Theta}_W^{(k)}(t)$ as the corresponding \emph{queue backlogs} 
at the latest $W$ type-$k$ exploration events before time $t$.  Define an estimate $\tilde{e}_k(t)$ 
as follows: 
\[  \tilde{e}_k(t) \defequiv \frac{1}{W}\sum_{w=1}^{W} \min_{I\in\script{I}}\left[Y_k(I, \bv{\omega}_{w}^{(k)}(t), \bv{\Theta}_w^{(k)}(t))  \right]  \]
The $\tilde{e}_k(t)$ estimate is adjusted appropriately if fewer than $W$ type-$k$ exploration events have
occurred (being set to zero initially).   This approach is different from Approach 1 in that 
the current queue backlogs are not used.  Hence, this is simply an empirical average over the past $W$
samples of the actual cost achieved in the $I^{mw}(\tau)$ computation (\ref{eq:I-mw}) at those particular 
sample times $\tau$.  Because
$I^{mw}(\tau)$ (and its corresponding cost) was already computed on slot $\tau$ in order to make the stage-2 control decision,
we can simply reuse the same value, without requiring any additional computation of problems
of type (\ref{eq:I-mw}). 

\subsection{The Max-Weight Learning Algorithm} 

Let $\theta$ be a given exploration probability (so that $0 \leq \theta < 1$
and exploration events of type $K$ occur with probability $\theta/K$). 
Let $\sigma>0$ be a given parameter, and let $V(t)$ be a 
given (non-negative) control function of slot $t$ (possibly a constant
function).  Let $\hat{W}(t)$ be a (possibly constant)  
function such that $\hat{W}(t) \geq 1$ for all $t$, and  define $W_0 \defequiv \hat{W}(0)$.
Define the
actual window size used at slot $t$ (for either Approach 1 or Approach 2) as follows: 
\[ W(t) \defequiv \min[\hat{W}(t), W_{rand}(t)] \]
where $W_{rand}(t)$ is the minimum number exploration events that
have occurred for any type (minimized over the types $k \in \{1, \ldots, K\}$), including
the $W_0$ events that take place at initialization as described below. 
Thus, there are always at least $W(t)$ type-$k$ exploration events 
by time $t$.  The  \emph{Max-Weight Learning Algorithm} is as follows.

\begin{itemize} 

\item (Initialization)  For a given integer $W_0>0$, let $\bv{\Theta}(-KW_0)=\bv{0}$, and
run the system over slots
$t = \{-W_0K, -W_0K + 1, \ldots, -1\}$, choosing each stage-1 decision option
$k \in \{1, \ldots, K\}$ in a fixed round-robin order (and choosing $I^{mw}(t)$  
according to (\ref{eq:I-mw}) and 
$\bv{\gamma}^{mw}(t)$ according to (\ref{eq:gamma-opt-mw})-(\ref{eq:gamma-constraint-mw})).
This ensures that we have  $W_0$ independent samples by time $0$, 
and creates a possibly non-zero initial queue state $\bv{\Theta}(0)$. 
Next perform the following sequence of actions for each slot $t\geq0$.

\item (Stage-1 Decisions) 
Independently with probability $\theta$, decide to have an exploration event.  If there is an exploration
event, choose $k(t)$ uniformly over all options $\{1, \ldots, K\}$.  If there is no exploration event, 
then under Approach 1 we observe current queue backlogs $\bv{\Theta}(t)$ and 
compute $\hat{e}_k(t)$ for each $k \in \{1, \ldots, K\}$ (using window size $W(t)$).  
We then choose $k(t)$ as the index $k \in \{1, \ldots, K\}$
that  minimizes $\hat{e}_k(t)$ (breaking ties arbitrarily). 
Under Approach 2, if there is no exploration event  we choose $k(t)$
to minimize $\tilde{e}_k(t)$ (using window size $W(t)$).

\item (Stage-2 Decisions)  Observe the queue backlogs $\bv{\Theta}(t)$ and the outcome $\bv{\omega}(t)$ that
resulted from the stage-1 decision.   Then choose $I^{mw}(t) \in \script{I}$ according to (\ref{eq:I-mw}). 
Choose auxiliary variables $\bv{\gamma}^{mw}(t)$ 
according to (\ref{eq:gamma-opt-mw})-(\ref{eq:gamma-constraint-mw}). 

\item (Past Value Storage) For Approach 1, store the resulting $\bv{\omega}(t)$ vector in memory 
as appropriate. For Approach 2, 
store the resulting cost from  (\ref{eq:I-mw})  in memory as appropriate. 

\item (Queue Updates)  Update virtual queues $U_n(t)$ according to (\ref{eq:u-dynamics}) and $Z_m(t)$
according to (\ref{eq:z-dynamics}).  Also allow the actual system queues $Q_l(t)$
to proceed according to (\ref{eq:q-dynamics}). 
\end{itemize} 

\emph{Remark 3:}  For some systems, we may not require an exploration event for each of the $K$ stage-1 decision
options.  For example, in an $L$-queue downlink where the decisions are to either measure all channels, blindly transmit
over one of the $L$ channels, or remain idle (as in \cite{chih-ping-channel-measure}), 
there are $K=L+2$ stage-1 options.  However, the ``idle'' choice does
not require any exploration events, as it clearly incurs a cost of $0$.  Further, the information gained by randomly choosing
to blindly transmit over a given channel can also be gained by measuring all channels, as the outcome
of the channel measurement can be used to determine if a blind transmission would have been successful.   
 It is therefore more
efficient to modify the algorithm by considering 
only  \emph{one type of exploration event}: the one that  randomly chooses to measure all channels.  Similarly, in 
DIVBAR-like situations where the $K$ decisions involve sending a packet of one of the various commodities (as in 
\cite{neely-divbar-journal}), 
 the success/failure event observed after sending any particular packet does not depend on the packet commodity
 and hence can be used to update the max-weight estimates for each commodity.

\subsection{Analysis of the Max-Weight Learning Algorithm} \label{section:learning-analysis}

For brevity, we analyze only Approach 2.\footnote{Bounds on the performance of Approach 1 
can be obtained similarly. In practice, Approach 1 would typically have 
superior performance because it uses current queue backlogs.}   Let 
$k^{mw}(t)$ denote the (ideal)  max-weight stage-1 decision on slot $t$, and let $\tilde{k}(t)$
denote the Approach 2 decision.  Recall that Approach 2 also uses the (ideal) $I^{mw}(t)$ 
and $\bv{\gamma}^{mw}(t)$ decisions.   Our goal is to compute parameters $C$, $\epsilon_V$, 
$\epsilon_U$, $\epsilon_Z$, $\epsilon_Q$  for (\ref{eq:approx-assumption}) 
that can be plugged into Theorem \ref{thm:1}.

\begin{thm} \label{thm:2} (Performance Under Approach 2 --- Fixed Window)  Suppose the Max-Weight
Learning Algorithm with Approach 2 is implemented using an exploration probability $\theta>0$.
Suppose we use  
a  fixed integer window size $W = W_0>0$ (so that $W(t) = W$ for all $t$, and our initialization
takes $W$ 
samples from each exploration type before time $0$).  Suppose that $V(t)$ is held constant, 
so that $V(t) = V$ for some $V>0$.  Then condition (\ref{eq:approx-assumption}) of Assumption 
3 holds with: 
\begin{eqnarray*}
C =  \frac{c WK^2(1+\theta)}{\theta}  \: \:  , \: \: 
\epsilon_V =  \epsilon_U =   \epsilon_Z =  \epsilon_Q = \frac{Ky_{diff}^{max}}{2\sqrt{W}}
\end{eqnarray*}
where $c$ and $y_{diff}^{max}$ are constants that are independent of queue backlog and of $V$,  $W$, $\theta$
(and depend on the  maximum and minimum penalties and maximum queue changes 
 that can occur  on one slot).
\end{thm}
\begin{proof}
See Appendix C. 
\end{proof}

It follows that if the fixed window size $W$ is chosen to be suitably large, then the $\epsilon_V$, $\epsilon_U$, $\epsilon_Z$, 
$\epsilon_Q$ constants will be small enough to satisfy the conditions $\epsilon_U < \epsilon_{max}$, 
$\epsilon_Z < \sigma$, $\epsilon_Q  < \epsilon_{max}$ required for Theorem \ref{thm:1}, and hence the 
result of Theorem \ref{thm:1} holds for this max-weight learning algorithm. 

\begin{thm}  \label{thm:variable-w} (Performance Under Approach 2 --- Variable $W(t)$ and $V(t)$) 
Suppose that we use the Max-Weight Learning algorithm (with Approach 2) using an
exploration probability $\theta>0$ and a 
variable $V(t)$ and $W(t)$  with initialization parameter $W_0 = 1$, and with: 
\[ V(t) = (t+1)^{\beta_2}V_0 \: \: , \: \: W(t) = \min[(t+1)^{\beta_1}, W_{rand}(t)] \]
where $\beta_1$ and $\beta_2$ are constants such that $0 < \beta_1 < \beta_2<1$, $V_0$ is a positive
constant, and 
where we recall that $W_{rand}(t)$ is the minimum number exploration events of type $k$ that have occurred,
minimized over all $k \in \{1, \ldots, K\}$. 
Then the time average constraints (\ref{eq:st2})-(\ref{eq:st2-aux}) hold, all queues $Q_l(t)$ are \emph{mean rate stable}, 
and the time average cost converges to the optimal value $f_{\theta}^*$: 
\[ \lim_{t\rightarrow\infty} f(\overline{\bv{x}}(t)) = f_{\theta}^* \]
\end{thm}
\begin{proof}
The proof combines results from the proofs of Theorems \ref{thm:2} and \ref{thm:variable-v}, and
is given in Appendix E.  
\end{proof}

\section{Conclusion} 

This work extends the important max-weight framework for stochastic network optimization 
to a context with 2-stage decisions and unknown distributions that govern the stochastics
at the first stage. This is useful in a variety of contexts, including transmission scheduling
in wireless networks in  unknown environments and with unknown channels.  The learning
algorithms developed here are based on estimates of expected 
max-weight functionals, and are much more efficient than algorithms that would 
attempt to learn the complete probability distributions associated with the system. 
Our analysis provides explicit bounds on the deviation from optimality in terms of the 
sample size $W$ and the control parameter $V$.  The $W$ and $V$ 
parameters also affect an explicit tradeoff in average
congestion and delay.  A modified algorithm with time-varying $W(t)$ and $V(t)$ parameters
was shown to converge to exact optimal performance while keeping all queues mean-rate
stable, at the cost of incurring a possibly infinite average congestion and delay.


\section*{Appendix A --- Proof of Theorem \ref{thm:1}} 

\begin{proof} (Theorem \ref{thm:1} --- The Queue Stability Inequality (\ref{eq:thm1-q}))
Writing the drift inequality (\ref{eq:main-drift}) using the $RHS(\cdot)$ function yields:
\begin{eqnarray*}
\expect{Vl(\bv{x}(t)) + V\tilde{f}(\bv{\gamma}(t)) \left|\right.\bv{\Theta}(t)} + \Delta(\bv{\Theta}(t)) \leq 
\expect{RHS(t, \bv{\Theta}(t), k(t), I(t), \bv{\gamma}(t)) \left|\right. \bv{\Theta}(t)} 
\end{eqnarray*}
Taking expectations of both sides with respect to the queue state distribution for $\bv{\Theta}(t)$ and
using the law of iterated expectations yields: 
\begin{eqnarray}
V\expect{l(\bv{x}(t)) + \tilde{f}(\bv{\gamma}(t))} + \expect{L(\bv{\Theta}(t+1))} - \expect{L(\bv{\Theta}(t))} &\leq& 
\expect{RHS(t, \bv{\Theta}(t), k(t), I(t), \bv{\gamma}(t))} \nonumber \\
&\leq& \expect{RHS(t, \bv{\Theta}(t), k^{mw}(t), I^{mw}(t), \bv{\gamma}^{mw}(t))} \nonumber \\
&& + C + V\epsilon_V + \epsilon_Z\sum_{m\in\tilde{\script{M}}} \expect{ |Z_m(t)|} \nonumber \\
&& + \epsilon_U \sum_{n=1}^N\expect{U_n(t)} + \epsilon_Q\sum_{l=1}^L\expect{Q_l(t)} \label{eq:appa1} \\
&\leq& \expect{RHS(t, \bv{\Theta}(t), k'(t), I'(t), \bv{\gamma}'(t))} \nonumber \\
&& + C + V\epsilon_V + \epsilon_Z \sum_{m\in\tilde{\script{M}}} \expect{|Z_m(t)|} \nonumber \\
&& + \epsilon_U \sum_{n=1}^N\expect{U_n(t)} + \epsilon_Q\sum_{l=1}^L\expect{Q_l(t)} \label{eq:appa2} 
\end{eqnarray}
where  (\ref{eq:appa1})  holds by Assumption 3, and (\ref{eq:appa2})  holds because
the max-weight policy minimizes the expectation of $RHS(\cdot)$ over all alternative decisions
for slot $t$.  The decisions $k'(t)$, $I'(t)$, $\bv{\gamma}'(t)$ can be chosen as any feasible control
decisions for slot $t$ (where a feasible control decision for $k'(t)$ must also respect 
the random exploration events of probability $\theta$).  
Suppose that $k'(t)$ and $I'(t)$ are the decisions  given in Assumption 2, so that properties 
(\ref{eq:prime1}) and (\ref{eq:prime2}) hold.   Choose auxiliary decision variables 
$\bv{\gamma}'(t) = (\gamma_m'(t))_{m\in\tilde{\script{M}}}$  as follows: 
\begin{equation} \label{eq:gamma-prime-def} 
  \gamma_m'(t)= \left\{ \begin{array}{ll}
                         \expect{x_m'(t)}   + \sigma   &\mbox{ if $Z_m(t) \geq 0$} \\
                             \expect{x_m'(t)} - \sigma  & \mbox{ if $Z_m(t) < 0$} 
                            \end{array}
                                 \right.  
  \end{equation} 
Note that these $\gamma_m'(t)$ decisions satisfy the required constraints
(\ref{eq:gamma-constraint}).  That is because 
for each $m \in \tilde{\script{M}}$ we have $x_m^{min} \leq  \expect{x_m'(t)}  \leq x_m^{max}$  
and therefore: 
\[ x_m^{min} - \sigma   \leq \expect{x_m'(t)}  - \sigma \leq \expect{x_m'(t)}  +\sigma \leq x_m^{max} + \sigma \] 

Using these $\gamma_m'(t)$ decisions and 
the definition of $RHS(\cdot)$  in the inequality (\ref{eq:appa2}) 
yields:\footnote{Recall that $RHS(\cdot)$ is defined as 
the right hand side of (\ref{eq:main-drift}).}
\begin{eqnarray}
V\expect{l(\bv{x}(t)) + \tilde{f}(\bv{\gamma}(t))} + \expect{L(\bv{\Theta}(t+1))} - \expect{L(\bv{\Theta}(t))} &\leq& 
B + C  + V\epsilon_V \nonumber \\
+ \expect{V l(\bv{x}'(t)) + V\tilde{f}(\bv{\gamma}'(t))} \nonumber \\
- \sum_{m\in\tilde{\script{M}}} \expect{Z_m(t)[\expect{x_m'(t)}  - x_m'(t)]} \nonumber \\
- \sum_{m\in\tilde{\script{M}}} \expect{|Z_m(t)|[\sigma - \epsilon_Z]} \nonumber \\
- \sum_{n=1}^N \expect{U_n(t)[b_n  - h_n(\bv{x}'(t))  - \epsilon_U]} \nonumber \\
- \sum_{l=1}^L \expect{Q_l(t)[\mu_l'(t) - A_l'(t) - \epsilon_Q]} \label{eq:rhs2} 
\end{eqnarray}
Note that because the policies $k'(t)$ and $I'(t)$ are stationary, randomized, and independent of the 
queue backlog vector $\bv{\Theta}(t)$, and because the functions $h_n(\bv{x})$  are linear or affine, we have: 
\begin{eqnarray*}
\expect{U_n(t)h_n(\bv{x}'(t))} &=& \expect{U_n(t)} h_n(\expect{\bv{x}'(t)})  \\
\expect{Z_m(t) x_m'(t)} &=& \expect{Z_m(t)}\expect{x_m'(t)} \\ 
 \expect{Q_l(t)[\mu_l'(t)- A_l'(t)]} &=& \expect{Q_l(t)}\expect{\mu_l'(t)-A_l'(t)} 
\end{eqnarray*}
Using these identities together with properties (\ref{eq:prime1})-(\ref{eq:prime2}) 
directly in the right hand side of (\ref{eq:rhs2}) and rearranging terms yields: 
\begin{eqnarray}
 \expect{L(\bv{\Theta}(t+1))} - \expect{L(\bv{\Theta}(t))} &\leq& 
B + C + V[l_{diff} + \tilde{f}_{diff} + \epsilon_V] \nonumber \\
&&-(\epsilon_{max}  - \epsilon_U)\sum_{n=1}^N\expect{U_n(t)} \nonumber \\
&& - (\sigma - \epsilon_Z)\sum_{m \in \tilde{\script{M}}} \expect{|Z_m(t)|} \nonumber \\
&& -(\epsilon_{max} - \epsilon_Q)\sum_{l=1}^L\expect{Q_l(t)} \label{eq:rhs3} 
\end{eqnarray}
where we have used the following fact: 
\[ \expect{l(\bv{x}'(t)) - l(\bv{x}(t))} \leq l_{diff} \: \: , \: \: \expect{\tilde{f}(\bv{\gamma}') - \tilde{f}(\bv{\gamma}(t))} \leq \tilde{f}_{diff} \]

The  inequality (\ref{eq:rhs3}) holds for all slots $t \in \{0, 1, 2 ,\ldots\}$. Summing the telescoping series
over $\tau \in \{0, 1, \ldots, t-1\}$ (as in \cite{now})  and dividing by $t$ yields: 
\begin{eqnarray*}
\frac{\expect{L(\bv{\Theta}(t))} - \expect{L(\bv{\Theta}(0))}}{t} \leq B + C + V[l_{diff} + \tilde{f}_{diff} + \epsilon_V] \\
- \frac{1}{t}\sum_{\tau=0}^{t-1}\left[(\epsilon_{max} - \epsilon_U)\sum_{n=1}^N \expect{U_n(\tau)} + (\sigma - \epsilon_Z) \sum_{m\in\tilde{\script{M}}} \expect{|Z_m(\tau)|} + (\epsilon_{max} - \epsilon_Q)\sum_{l=1}^L \expect{Q_l(\tau)} \right]  
\end{eqnarray*}
Using non-negativity of the Lyapunov function $L(\cdot)$ in the above inequality proves (\ref{eq:thm1-q}).
Taking the $\limsup$ of (\ref{eq:thm1-q}) as $t \rightarrow \infty$ proves that the queues $Q_l(t)$, $Z_m(t)$, 
$U_n(t)$ are strongly stable (for all $l \in \{1, \ldots, L\}$, $m \in \tilde{\script{M}}$, $n \in \{1, \ldots, N\}$). 
 Hence (by Lemma \ref{lem:q-stable}), the inequality constraints (\ref{eq:st2})-(\ref{eq:stability-c2}) 
 are satisfied. 
\end{proof}

\begin{proof} (Theorem \ref{thm:1} --- The Utility Inequality   (\ref{eq:thm1-cost}))
Recall that the inequality 
(\ref{eq:appa2}) holds for any alternative set of feasible control decisions 
$k''(t)$, $I''(t)$, $\bv{\gamma}''(t)$.  Re-writing (\ref{eq:appa2}) using this notation and using
the definition of $RHS(\cdot)$ yields: 
\begin{eqnarray*}
V\expect{l(\bv{x}(t)) + \tilde{f}(\bv{\gamma}(t))} + \expect{L(\bv{\Theta}(t+1))} - \expect{L(\bv{\Theta}(t))} &\leq&
B + C + V\epsilon_V + \epsilon_Z\sum_{m\in\tilde{\script{M}}} \expect{|Z_m(t)|}  \\
&& + \expect{V l(\bv{x}''(t)) + V\tilde{f}(\bv{\gamma}''(t))} \\
&& - \sum_{n=1}^N\expect{U_n(t)(b_n  - h_n(\bv{x}''(t)) - \epsilon_U)} \\
&& - \sum_{m\in\tilde{M}}\expect{Z_m(t)(\gamma_m''(t) - x_m''(t))} \\
&& - \sum_{l=1}^L \expect{Q_l(t)(\mu_l''(t) - A_l''(t) - \epsilon_Q)} 
\end{eqnarray*}
Let $\alpha$ be a probability (to be chosen later), and 
define joint control actions $(k''(t); I''(t); \bv{\gamma}''(t))$ as follows: 
\begin{eqnarray*}
(k''(t); I''(t); \bv{\gamma}''(t))   = \left\{ \begin{array}{ll}
                          (k'(t); I'(t); \bv{\gamma}'(t))  &\mbox{ with prob. $\alpha$} \\
                             (k^*(t); I^*(t); \bv{\gamma}^*) & \mbox{ with prob. $1-\alpha$} 
                            \end{array}
                                 \right. 
\end{eqnarray*}
where $k'(t)$,  $I'(t)$
are as defined in Assumption 2 (and satisfy (\ref{eq:prime1})-(\ref{eq:prime2})), variables
$\gamma_m'(t)$ are as defined in (\ref{eq:gamma-prime-def}), 
and $I^*(t)$, $k^*(t)$, $\bv{\gamma}^*$ are as defined in Assumption 1 (and satisfy properties 
(\ref{eq:ass1-1})-(\ref{eq:ass1-3})).   Note that the $k''(t)$ decision defined here still has random
exploration events with probability $\theta$, as both $k'(t)$ and $k^*(t)$ have such events. 
Also note that $\gamma_m''(t)$ satisfies (\ref{eq:gamma-constraint})
because both $\gamma_m'(t)$ and $\gamma_m^*$ satisfy (\ref{eq:gamma-constraint}).  
Further, we have: 
\begin{eqnarray*}
\expect{\bv{x}''(t)}  &=&  \alpha\expect{\bv{x}'(t)}  + (1-\alpha)\expect{\bv{x}^*(t)}  \\
\expect{\bv{\gamma}''(t)} &=& \alpha\expect{\bv{\gamma}'(t)} + (1-\alpha) \bv{\gamma}^*
\end{eqnarray*}

  It follows from  properties (\ref{eq:prime1})-(\ref{eq:prime2}) and 
  (\ref{eq:ass1-1})-(\ref{eq:ass1-3})
  (together with linearity of 
$l(\bv{x})$ and $h_n(\bv{x})$ and the fact that the
randomized $k''(t)$ and $I''(t)$ choices are independent of queue backlog) that:
\begin{eqnarray*}
V\expect{l(\bv{x}(t)) + \tilde{f}(\bv{\gamma}(t))} + \expect{L(\bv{\Theta}(t+1))} - \expect{L(\bv{\Theta}(t))} &\leq&
B + C + V\epsilon_V \\
&&  + (1-\alpha)Vl(\expect{\bv{x}^*(t)}) +  (1-\alpha)V \tilde{f}(\bv{\gamma}^*) \\
&& + \alpha Vl(\expect{\bv{x}'(t)}) + \alpha V \expect{\tilde{f}(\bv{\gamma}'(t))}\\
&& - \sum_{n=1}^N \expect{U_n(t)}(\alpha\epsilon_{max}  
- \epsilon_U) \\
&& - \sum_{m \in\tilde{\script{M}}} \expect{|Z_m(t)|}(\alpha \sigma - \epsilon_Z) \\
&& - \sum_{l=1}^L \expect{Q_l(t)}(\alpha \epsilon_{max} - \epsilon_Q)
\end{eqnarray*}
Now choose $\alpha$ as follows: 
\[ \alpha = \max\left[  \frac{\epsilon_U}{\epsilon_{max}},  \frac{\epsilon_Z}{\sigma}, 
\frac{\epsilon_Q}{\epsilon_{max}}\right] \]
This is a valid probability because we have assumed that $\epsilon_U < \epsilon_{max}$, 
$\epsilon_Z < \sigma$, $\epsilon_Q < \epsilon_{max}$. 
The above inequality  reduces to: 
\begin{eqnarray}
V\expect{l(\bv{x}(t)) + \tilde{f}(\bv{\gamma}(t))} + \expect{L(\bv{\Theta}(t+1))} - \expect{L(\bv{\Theta}(t))} \leq \nonumber \\
B + C + V\epsilon_V 
 + Vf^*_{\theta} + \alpha V (l_{diff} + \tilde{f}_{diff}) \label{eq:utility-reuse} 
\end{eqnarray}
The above inequality holds for all $t$.  Taking a telescoping series over $\tau \in \{0, 1, \ldots, t-1\}$
yields: 
\begin{eqnarray*}
\frac{1}{t}\sum_{\tau=0}^{t-1} \expect{l(\bv{x}(\tau)) + \tilde{f}(\bv{\gamma}(\tau))} + \frac{\expect{L(\bv{\Theta}(t))} - \expect{L(\bv{\Theta}(0))}}{Vt} \leq f^*_{\theta} + \epsilon_V + \alpha (l_{diff} + \tilde{f}_{diff}) + \frac{B + C}{V} 
\end{eqnarray*}
Therefore, using $\delta \defequiv \alpha(l_{diff} + \tilde{f}_{diff})$, non-negativity of $L(\cdot)$, and
Jensen's inequality with 
convexity of $l(\bv{x})$ and $\tilde{f}(\bv{x})$,  we have: 
\begin{equation*} 
 l(\overline{\bv{x}}(t)) + \tilde{f}(\overline{\bv{\gamma}}(t)) \leq f^*_{\theta} + \epsilon_V + \delta + \frac{B+C}{V} + \frac{\expect{L(\bv{\Theta}(0))}}{Vt} 
\end{equation*} 
However, we have: 
\[ \tilde{f}(\overline{\bv{\gamma}}(t)) \geq \tilde{f}\left(\overline{\tilde{\bv{x}}}(t)\right) - 
 \tilde{M}\nu\norm{\overline{\bv{\tilde{x}}}(t) - \overline{\bv{\gamma}}(t)} \]
where $\nu$ is the magnitude of the largest left or right 
partial derivative of the $\tilde{f}(\cdot)$ function and $\tilde{M}$ is the cardinality of $\tilde{\script{M}}$.\footnote{Left and right partial derivatives exist and are
finite for any
convex function that is defined over the full space $\mathbb{R}^M$.}
Combining the above two inequalities and using the fact that $f(\bv{x}) \defequiv l(\bv{x}) + \tilde{f}(\tilde{\bv{x}})$
yields: 
\begin{equation} \label{eq:utility-proof}
 f(\overline{\bv{x}}(t)) - \tilde{M}\nu\norm{\overline{\bv{\tilde{x}}}(t) - \overline{\bv{\gamma}}(t)} \leq f^*_{\theta} + \epsilon_V + \delta + \frac{B+C}{V} + \frac{\expect{L(\bv{\Theta}(0))}}{Vt} 
\end{equation} 
Because the equality constraints (\ref{eq:st2-aux}) hold, we have that 
$\norm{\overline{\bv{\tilde{x}}}(t) - \overline{\bv{\gamma}}(t)} \rightarrow 0$.
Taking the $\limsup$ of (\ref{eq:utility-proof}) as $t \rightarrow \infty$ 
thus yields (\ref{eq:thm1-cost}), completing the proof. 
\end{proof}

Note that in the special case when there are no auxiliary variables (so that $f(\bv{x})$ is linear and 
$f(\bv{x}) = l(\bv{x})$), and when all queues are initially empty, the inequality (\ref{eq:utility-proof}) reduces to
the following cost guarantee that holds for all time $t$: 
\[ f(\overline{\bv{x}}(t)) \leq f^*_{\theta} + \epsilon_V + \delta + (B+C)/V \]

\section*{Appendix B ---  Proof of the Variable $V(t)$ Theorem (Theorem \ref{thm:variable-v})} 

\begin{proof} (Mean Rate Stability of all Queues) 
Assume without loss of generality that $\epsilon_U(t) < \epsilon_{max}$, $\epsilon_Z(t) < \sigma$, 
$\epsilon_Q(t) < \epsilon_{max}$ for all $t \geq t_0$ (else, choose a time $\tilde{t}_0$ for which this holds). Then,
on a single slot $t$, we can apply the result from the proof of Theorem \ref{thm:1} with $V \defequiv V(t)$ and 
$\epsilon_x \defequiv \epsilon_x(t)$ (for $x \in \{V, U, Z, Q\}$).  Thus, for any time $t \geq t_0$ we 
have from (\ref{eq:rhs3}): 
\[ \expect{L(\bv{\Theta}(t+1))} - \expect{L(\bv{\Theta}(t))} \leq B + C(t) + V(t)[l_{diff} + \tilde{f}_{diff} + \epsilon_V(t)] \]
where we have neglected the three non-positive terms on the right hand side of (\ref{eq:rhs3}). 
Summing the above inequality over $\tau \in \{t_0, \ldots, t-1\}$ yields: 
\[ \frac{\expect{L(\bv{\Theta}(t))} - \expect{L(\bv{\Theta}(t_0))}}{t-t_0} \leq B + \frac{O(t^{\beta_2 + 1})}{t-t_0} \] 
where we have used the fact that $\sum_{\tau=t_0}^{t-1} C(\tau) \leq O(t^{\beta_1 + 1})$ and 
$\sum_{\tau=t_0}^{t-1} V(\tau) \leq O(t^{\beta_2 + 1})$.  Because $L(\bv{\Theta}(t))$ is a sum of squared queue
lengths (for all queues), the above inequality  implies that  for any queue $Q_l(t)$: 
\[ \frac{\expect{Q_l(t)^2}}{t-t_0} \leq B + \frac{O(t^{\beta_2+1})}{t-t_0} + \frac{\expect{L(\bv{\Theta}(t_0))}}{t-t_0} \]
Dividing the above inequality by $t-t_0$, taking square roots, and using the fact that $\expect{Q_l(t)^2} \geq \expect{Q_l(t)}^2$
yields: 
\[  \frac{\expect{Q_l(t)}}{t-t_0} \leq \sqrt{\frac{B}{(t-t_0)} + \frac{O(t^{\beta_2+1})}{(t-t_0)^2} + \frac{\expect{L(\bv{\Theta}(t_0))}}{(t-t_0)^2}} \]
Because $\beta_2 + 1 < 2$, the right hand side above converges to $0$ as $t \rightarrow\infty$.  This holds for all 
queues $Q_l(t)$, and hence all these queues are \emph{mean rate stable}.  Similarly, it holds for all queues
$Z_m(t)$ and $U_n(t)$ (for $m \in \tilde{\script{M}}$ and $n \in \{1, \ldots, N\}$), and so all these queues are mean
rate stable.  It follows by Lemma \ref{lem:q-stable} that all inequality constraints (\ref{eq:st2})-(\ref{eq:st2-aux}) are satisfied. 
\end{proof} 

\begin{proof} (Cost Optimality) 
Again assume (without loss of generality) that $\epsilon_U(t) < \epsilon_{max}$, $\epsilon_Z(t) < \sigma$, 
$\epsilon_Q(t) < \epsilon_{max}$ for all $t \geq t_0$.  We thus have from (\ref{eq:utility-reuse}) that: 
\begin{eqnarray*}
\expect{l(\bv{x}(t)) + \tilde{f}(\bv{\gamma}(t))} + \frac{\expect{L(\bv{\Theta}(t+1)) - L(\bv{\Theta}(t))}}{V(t)} \leq 
\frac{B + C(t)}{V(t)}  + \epsilon_V(t) + f_{\theta}^* + \alpha(t)(l_{diff} - \tilde{f}_{diff})
\end{eqnarray*}
where $\alpha(t)$ is defined: 
\[ \alpha(t) \defequiv \max\left[\frac{\epsilon_U(t)}{\epsilon_{max}}, \frac{\epsilon_Z(t)}{\sigma}, \frac{\epsilon_Q(t)}{\epsilon_{max}}\right] \]
and satisfies $\alpha(t) \rightarrow 0$ as $t \rightarrow \infty$. 
The above holds for all $t \geq t_0$.  Summing over $\tau \in \{t_0, \ldots, t-1\}$ yields: 
\begin{eqnarray*}
\sum_{\tau=t_0}^{t-1} \expect{l(\bv{x}(\tau)) + \tilde{f}(\bv{\gamma}(\tau))} + \sum_{\tau=t_0+1}^{t-1}\expect{L(\bv{\Theta}(\tau))}\left[\frac{1}{V(\tau-1)} - \frac{1}{V(\tau)}\right] + \frac{\expect{L(\bv{\Theta}(t))}}{V(t-1)} - \frac{\expect{L(\bv{\Theta}(t_0))}}{V(t_0)}  \leq  \\
 (t-t_0) f_{\theta}^*   + 
\sum_{\tau=t_0}^{t-1} \left[ \frac{B + C(\tau)}{V(\tau)} + \epsilon_V(\tau) + \alpha(\tau)(l_{diff} - \tilde{f}_{diff})   \right] 
\end{eqnarray*}
Using non-negativity of $L(\cdot)$ and the fact that $\frac{1}{V(\tau-1)} - \frac{1}{V(\tau)} \geq 0$ (because $V(\tau)$ is non-decreasing),
and dividing by $(t-t_0)$ yields: 
\begin{eqnarray}
\frac{1}{t-t_0}\sum_{\tau=t_0}^{t-1} \expect{l(\bv{x}(\tau)) + \tilde{f}(\bv{\gamma}(\tau))} - \frac{\expect{L(\bv{\Theta}(t_0))}}{(t-t_0)V(t_0)} \leq 
f_{\theta}^*  +   \Psi(t)  \label{eq:appb-buz} 
\end{eqnarray}
where $\Psi(t)$ is defined: 
\[ \Psi(t) \defequiv 
\frac{1}{t-t_0} 
\sum_{\tau=t_0}^{t-1} \left[ \frac{B + C(\tau)}{V(\tau)} + \epsilon_V(\tau) + \alpha(\tau)(l_{diff} - \tilde{f}_{diff})   \right]  \]
Note that $C(\tau)/V(\tau) \rightarrow 0$ as $\tau \rightarrow \infty$, and hence  $\Psi(t)$ is the time average of a function that converges to $0$.  
We  thus have $\Psi(t) \rightarrow 0$ as $t \rightarrow 0$. 
By Jensen's inequality applied to the left hand side of (\ref{eq:appb-buz}) we have: 
\[ l(\overline{\bv{x}}(t)) + \tilde{f}(\overline{\bv{\gamma}}(t)) -  \frac{\expect{L(\bv{\Theta}(t_0))}}{(t-t_0)V(t_0)} \leq f_{\theta}^* + \Psi(t) \]
where $\overline{\bv{x}}(t)$ and $\overline{\bv{\gamma}}(t)$ are time average expectations over the interval $\tau \in \{t_0, \ldots, t-1\}$.
Because we already know $Z_m(t)$ is mean rate stable for all $m \in \tilde{\script{M}}$, we have that 
$\norm{\overline{\bv{\gamma}}(t) - \overline{\tilde{\bv{x}}}(t)} \rightarrow 0$ as $t \rightarrow \infty$ (by Lemma \ref{lem:q-stable}), 
and hence, as in the proof of Theorem \ref{thm:1}
(using $f(\bv{x}) = l(\bv{x}) + \tilde{f}(\tilde{\bv{x}})$): 
\[ \limsup_{t\rightarrow\infty} f(\overline{\bv{x}}(t)) \leq f_{\theta}^* \]
Because $f_{\theta}^*$ is defined as the infimum cost subject to queue stability,\footnote{It can be shown that the infimum cost subject
to \emph{strong stability} is the same as the infimum cost subject to \emph{mean rate stability}.}  
it can be shown that the $\liminf$ cannot be lower than $f_{\theta}^*$, and 
so the limit of $f(\overline{\bv{x}}(t))$ exists and is equal to the $\limsup$, proving the result.   
\end{proof}

\section*{Appendix C --- Proof of Theorem \ref{thm:2}} 

To prove Theorem \ref{thm:2}, fix time $t$ and 
define $\Omega(\bv{\Theta}(t))$ as follows: 
\begin{eqnarray*}
\Omega(\bv{\Theta}(t)) \defequiv \expect{RHS(t, \bv{\Theta}(t), \tilde{k}(t), I^{mw}(t), \bv{\gamma}^{mw}(t))\left|\right.\bv{\Theta}(t)} \\
-  \expect{RHS(t, \bv{\Theta}(t), k^{mw}(t), I^{mw}(t), \bv{\gamma}^{mw}(t))\left|\right.\bv{\Theta}(t)}
\end{eqnarray*} 
Now note that because these right-hand sides differ only in terms comprising
the $e_k(t)$ expression, we have: 
\begin{eqnarray*}
\Omega(\bv{\Theta}(t)) = \expect{e_{\tilde{k}(t)}(t)\left|\right.\bv{\Theta}(t)}  - \min_{k\in\script{K}}[e_k(t)] 
\end{eqnarray*}
where the expectation on the right hand side is over the random decision
$\tilde{k}(t) = \arg\min_{k\in\script{K}}[\tilde{e}_k(t)]$, which is based on the empirical average $\tilde{e}_k(t)$ formed by  
the past $W$ random samples. It uses the fact that given a particular (possibly sub-optimal) decision
$\tilde{k}(t)\in\script{K}$, the resulting expected max-weight functional (using the exact but
unknown distribution function $F_{\tilde{k}(t)}(\bv{\omega})$) is $e_{\tilde{k}(t)}$.
Now for each $k \in \script{K}$, 
define $\delta_k(t) \defequiv \tilde{e}_k(t) - e_k(t)$.
We thus have: 
\begin{eqnarray*}
\expect{e_{\tilde{k}(t)}(t)\left|\right.\bv{\Theta}(t)}  &=& \expect{\tilde{e}_{\tilde{k}(t)}(t) - \delta_{\tilde{k}}(t)\left|\right.\bv{\Theta}(t)} \\
&\leq& \expect{\tilde{e}_{\tilde{k}(t)}(t)\left|\right.\bv{\Theta}(t)}  +  \expect{\max_{k \in \script{K}} [-\delta_{k}(t)] \left|\right.\bv{\Theta}(t)} \\
&=& \expect{\min_{k\in\script{K}}[\tilde{e}_k(t)] \left|\right.\bv{\Theta}(t)} + \expect{\max_{k\in\script{K}} [-\delta_k(t)]\left|\right.\bv{\Theta}(t)}  \\
&=& \expect{\min_{k \in \script{K}}[e_k(t) + \delta_k(t)]\left|\right.\bv{\Theta}(t)} +  \expect{\max_{k\in\script{K}}[-\delta_k(t)]\left|\right.\bv{\Theta}(t)} \\
&\leq& \min_{k\in\script{K}}[e_k(t)] + \expect{\max_{k\in\script{K}} [\delta_k(t)] +  \max_{k\in\script{K}} [-\delta_k(t)] \left|\right.\bv{\Theta}(t)} \\
&\leq& \min_{k\in\script{K}}[e_k(t)] + \sum_{k=1}^K \expect{\max[\delta_k(t), 0]  + \max[-\delta_k(t), 0]\left|\right.\bv{\Theta}(t)} \\
&=& \min_{k\in\script{K}}[e_k(t)] + \sum_{k=1}^K \expect{|\delta_k(t)|\left|\right.\bv{\Theta}(t)} 
\end{eqnarray*}
It follows that: 
\begin{equation*} 
 \Omega(\bv{\Theta}(t)) \leq \sum_{k=1}^K\expect{|\delta_k(t)|\left|\right.\bv{\Theta}(t)} 
 \end{equation*} 
 Therefore, by iterated expectations we have: 
 \begin{equation} \label{eq:puzzle0} 
  \expect{\Omega(\bv{\Theta}(t))} \leq \sum_{k=1}^K\expect{|\delta_k(t)|} 
  \end{equation} 

Note that $\expect{\Omega(\bv{\Theta}(t))}$ corresponds to the desired inequality (\ref{eq:approx-assumption}),
and hence it suffices to bound $\expect{|\delta_k(t)|}$.  
To this end, for each $k \in \{1, \ldots, K\}$, define $T_k(t)$ 
as the number of timeslots that passed after the $W$th-latest type $k$ exploration event.  Thus, all samples
$\bv{\omega}_w^{(k)}(t)$ occur on type-$k$ explorations events, and are on or after time $t - T_k(t)$.
Define $\bv{\Theta}_k(t) \defequiv \bv{\Theta}(t-T_k(t))$. We have: 
\begin{eqnarray}
|\delta_k(t)| &=& |\tilde{e}_k(t) - e_k(t)| \leq \left|\tilde{e}_k(t) - \tilde{e}_k^{prev}(t) \right| 
 + \left|   \tilde{e}_k^{prev}(t)  - e_k^{prev}(t)  \right| 
 + \left| e_k^{prev}(t) - e_k(t)  \right| \label{eq:puzzle1} 
\end{eqnarray}
where $\tilde{e}_k^{prev}(t)$ and $e_k^{prev}(t)$ are defined using queue lengths from  
the \emph{previous} time $t - T_k(t)$ as follows: 
\begin{eqnarray*}
\tilde{e}_k^{prev}(t) &\defequiv& \frac{1}{W}\sum_{w=1}^W \min_{I\in\script{I}}[Y_k(I, \bv{\omega}_w^{(k)}(t), \bv{\Theta}(t-T_k(t)))] \\
e_k^{prev}(t) &\defequiv& \expect{\min_{I\in\script{I}}\left[Y_k(I, \bv{\omega}(t), \bv{\Theta}(t-T_k(t)))\right]\left|\right.\bv{\Theta}(t-T_k(t)), k(t) = k}
\end{eqnarray*}
where the expectation in the definition of $e_k^{prev}(t)$ is with respect to the independent 
outcome $\bv{\omega}(t)$ that has distribution $F_k(\bv{\omega})$. 
 Comparing the definition of $e_k^{prev}(t)$ to the definition of $e_k(t)$ in (\ref{eq:ek}),
 it is clear that they are different only in that they use different queue states (similarly, 
  $\tilde{e}_k(t)$ and $\tilde{e}_k^{prev}(t)$ differ only in that they use different queue states).  
  Because the maximum change in queue size on any single slot is bounded, we have
the following lemma.
\begin{lem} For any $k \in \{1, \ldots, K\}$, any time $t$, and regardless of queue backlog $\bv{\Theta}(t)$, we have: 
\begin{equation} \label{eq:puzzle2} 
 \expect{  \left|\tilde{e}_k(t) - \tilde{e}_k^{prev}(t) \right| + \left| e_k^{prev}(t) - e_k(t)  \right|} \leq d_1 \expect{T_k(t)} 
 \end{equation}
where $d_1$ is a constant that is proportional to the maximum change in any queue over a single slot, and is independent
of the current queue sizes and of $W$ and $K$. 
\end{lem} 
\begin{proof} Define $I_w^{(k)}(t)$ as follows: 
 \[ I_w^{(k)}(t) \defequiv \arg\min_{I\in\script{I}}\left[ Y_k(I, \bv{\omega}_w^{(k)}, \bv{\Theta}(t-T_k(t)))  \right]   \]
 We have: 
\begin{eqnarray*}
\tilde{e}_k(t) &=& \frac{1}{W} \sum_{w=1}^{W} \min_{I \in \script{I}}\left[Y_k(I, \bv{\omega}_w^{(k)}(t), \bv{\Theta}_w^{(k)}(t))  \right] \\
&\leq&  \frac{1}{W} \sum_{w=1}^{W} Y_k(I_w^{(k)}(t), \bv{\omega}_w^{(k)}(t), \bv{\Theta}_w^{(k)}(t))  \\
&\leq&   \frac{1}{W} \sum_{w=1}^{W} Y_k(I_w^{(k)}(t), \bv{\omega}_w^{(k)}(t), \bv{\Theta}(t - T_k(t))) + c_1 T_k(t) \\
&=&  \tilde{e}_k^{prev}(t) + c_1 T_k(t)
\end{eqnarray*}
where $c_1$ is a constant that is proportional to the maximum change of any queue value over one slot. 
With an almost identical argument,  it can be shown that $\tilde{e}_k^{prev}(t) \leq \tilde{e}_k(t) + c_2 T_k(t)$, 
where $c_2$ is a constant that is proportional to the maximum change of any queue value over one slot. 
Thus: 
\[ |\tilde{e}_k(t) - \tilde{e}_k^{prev}(t)| \leq \max[c_1, c_2]T_k(t) \]
Therefore: 
\[ \expect{|\tilde{e}_k(t) - \tilde{e}_k^{prev}(t)|} \leq \max[c_1, c_2]\expect{T_k(t)} \]
Similarly, we can show: 
\[ \expect{|e_k(t) - e_k^{prev}(t)|}  \leq c_3 \expect{T_k(t)} \]
Defining $d_1\defequiv \max[c_1, c_2] + c_3$ proves the lemma. 
\end{proof} 

It now suffices to bound $\expect{|\tilde{e}_k^{prev}(t) - e_k^{prev}(t)|}$.  For a given $k \in \{1, \ldots, K\}$ and 
a given collection of queue states $\bv{\Theta}(t-T_k(t))$ at time $t-T_k(t)$, define the following function $Y(\bv{\omega})$: 
\begin{equation} \label{eq:random-variable} 
 Y(\bv{\omega})  \defequiv \min_{I \in \script{I}} \left[Y_k(I, \bv{\omega}, \bv{\Theta}(t- T_k(t))) \right]
 \end{equation} 
 Note that 
$\tilde{e}_k^{prev}(t)$ is simply an empirical average of the function $Y(\bv{\omega})$  over $W$  i.i.d. 
samples $\bv{\omega}_w^{(k)}(t)$ (which have distribution $F_k(\bv{\omega})$). 
Note that these values are also \emph{independent
of the queue state $\bv{\Theta}(t-T_k(t))$}, as these samples are taken on or 
after time $t-T_k(t)$.  Further, the value $e_k^{prev}(t)$ is simply an expected
value of the random variable $Y(\bv{\omega})$ 
over all outcomes $\bv{\omega}$ that take place with 
distribution $F_k(\bv{\omega})$. Hence we have reduced the 
problem to a pure ``Law of Large Numbers'' 
problem
of bounding the expected difference between the exact mean of a random variable and its empirical
average over $W$ i.i.d. samples.  Because the queue backlogs $\bv{\Theta}(t-T_k(t))$ are considered
constant in  $Y(\bv{\omega})$, we can write $Y(\bv{\omega})$ in terms of component
random variables  as follows (using  (\ref{eq:Y})): 
\begin{eqnarray}
Y(\bv{\omega}) =  \hspace{+5.5in}  \nonumber \\
\min_{I\in\script{I}} \left[ VY_V(\bv{\omega}) + \sum_{n=1}^NU_n(t-T_k(t))Y_{U,n}(\bv{\omega}) + \sum_{m\in\tilde{\script{M}}}Z_m(t-T_k(t))Y_{Z,m}(\bv{\omega})
- \sum_{l=1}^L Q_l(t-T_k(t))Y_{Q,l}(\bv{\omega}) \right] \label{eq:components} 
\end{eqnarray}
where $Y_V(\bv{\omega})$, $Y_{U,n}(\bv{\omega})$, $Y_{Z,n}(\bv{\omega})$, $Y_{Q, l}(\bv{\omega})$ are random variables defined as
(from (\ref{eq:Y})): 
\begin{eqnarray}
Y_V(\bv{\omega}) &\defequiv&  l(\hat{x}(k, \bv{\omega}, I^*_{\bv{\omega}})) \label{eq:Yv} \\
Y_{U,n}(\bv{\omega}) &\defequiv& h_n(\hat{x}(k, \bv{\omega}, I^*_{\bv{\omega}})) \\
Y_{Z, m}(\bv{\omega}) &\defequiv& \hat{x}_m(k, \bv{\omega}, I^*_{\bv{\omega}}) \\
Y_{Q,l}(\bv{\omega}) &\defequiv& \hat{\mu}_l(k, \bv{\omega}, I^*_{\bv{\omega}}) - \hat{a}_l(k, \bv{\omega}, I^*_{\bv{\omega}}) \label{eq:Yq} 
\end{eqnarray}
where $I^*_{\bv{\omega}}$ is the stage-2 control action that achieves the $\min$ in (\ref{eq:components}). 
Now define $\overline{Y}_V$, $\overline{Y}_{U, n}$, $\overline{Y}_{Z, m}$, $\overline{Y}_{Q, l}$
as the expectations of the random variables in (\ref{eq:Yv})-(\ref{eq:Yq}) over the random variable $\bv{\omega}$ that 
has distribution $F_k(\bv{\omega})$, and define $Y_V^{(W)}$, $Y_{U, n}^{(W)}$, $Y_{Z, m}^{(W)}$, $Y_{Q, l}^{(W)}$
as the corresponding \emph{empirical averages} over the i.i.d. samples $\bv{\omega}_w^{(k)}$ (for $w \in \{1, \ldots, W\}$). 
We thus have: 
\begin{eqnarray*}
\tilde{e}_k^{prev}(t) - e_k^{prev}(t) =  V(Y_V^{(W)} - \overline{Y}_V) +
 \sum_{n=1}^NU_n(t-T_k(t))(Y_{U,n}^{(W)} - \overline{Y}_{U,n}) \\
+ \sum_{m\in\tilde{\script{M}}} Z_m(t-T_k(t))(Y_{Z,m}^{(W)} - \overline{Y}_{Z,m})+ \sum_{l=1}^L Q_l(t-T_k(t))(Y_{Q,l}^{(W)} - \overline{Y}_{Q,l})
\end{eqnarray*}
and hence: 
\begin{eqnarray}
|\tilde{e}_k^{prev}(t) - e_k^{prev}(t)| \leq V|Y_V^{(W)} - \overline{Y}_V| + 
 \sum_{n=1}^NU_n(t-T_k(t))|Y_{U,n}^{(W)} - \overline{Y}_{U,n}| \nonumber \\
+ \sum_{m\in\tilde{\script{M}}} |Z_m(t-T_k(t))||Y_{Z,m}^{(W)} - \overline{Y}_{Z,m}|+ \sum_{l=1}^L Q_l(t)|Y_{Q,l}^{(W)} - \overline{Y}_{Q,l}|  \label{eq:almost-done}
\end{eqnarray}

We now use the following basic lemma concerning the expected difference between 
an empirical average and its exact mean: 
\begin{lem} \label{lem:lln}  Let $\{Y_w\}_{w=1}^{\infty}$ be an i.i.d. sequence of random variables with a general distribution with finite
support, so that there are finite constants $y_{min}$ and $y_{max}$ such that: 
\[ y_{min} \leq Y_w \leq y_{max} \: \: \mbox{for all $w \in \{1, 2, \ldots\}$} \]
Define $y_{diff} \defequiv y_{max} - y_{min}$. Define $\overline{Y}$ as the expectation of $Y_1$, and 
define $Y^{(W)}$ as the empirical average over $W$ samples: 
$Y^{(W)} \defequiv \frac{1}{W}\sum_{w=1}^{W} Y_w$.
Then: 
\[ \expect{|Y^{(W)} - \overline{Y}|} \leq  \frac{y_{diff}}{2\sqrt{W}} \]
\end{lem} 
\begin{proof} 
The proof is straightforward and is given in Appendix D for completeness. 
\end{proof}

Because all penalties and  cost functions are upper and lower bounded, 
the random variables in (\ref{eq:Yv})-(\ref{eq:Yq}) have finite support, and we define $y_{diff}^{max}$ 
as the maximum difference in the maximum and minimum possible values over all of the random
variables. Using Lemma \ref{lem:lln} in (\ref{eq:almost-done}) yields: 
\begin{eqnarray*}
\expect{|\tilde{e}_k^{prev}(t) - e_k^{prev}(t)|  \left|\right.\bv{\Theta}(t-T_k(t)), T_k(t)}  &\leq& \frac{y_{diff}^{max}}{2\sqrt{W}}\left[V  +
 \sum_{n=1}^NU_n(t-T_k(t))\right] \\
&& + \frac{y_{diff}^{max}}{2\sqrt{W}}\left[\sum_{m\in\tilde{\script{M}}} |Z_m(t-T_k(t))| + \sum_{l=1}^L Q_l(t-T_k(t))\right]  \\
&\leq&  \frac{y_{diff}^{max}}{2\sqrt{W}}\left[V  +
 \sum_{n=1}^NU_n(t)
 + \sum_{m\in\tilde{\script{M}}} |Z_m(t)| + \sum_{l=1}^L Q_l(t)\right]  \\
&&  +  d_2T_k(t)
\end{eqnarray*}
 where $d_2$ is a constant that depends on the maximum change in queue backlog on a given slot. 
 Taking expectations of the above 
 and using  the law of iterated expectations  yields: 
 \begin{eqnarray*}
 \expect{|\tilde{e}_k^{prev}(t) - e_k^{prev}(t)| } &\leq& \frac{y_{diff}^{max}}{2\sqrt{W}}\expect{V  +
 \sum_{n=1}^NU_n(t)
 + \sum_{m\in\tilde{\script{M}}} |Z_m(t)| + \sum_{l=1}^L Q_l(t)}  + d_2\expect{T_k(t)} 
 \end{eqnarray*}

Using the above inequality with (\ref{eq:puzzle2}), (\ref{eq:puzzle1}) in (\ref{eq:puzzle0}) yields: 
\begin{equation} \label{eq:finally-25} 
\expect{\Omega(\bv{\Theta}(t))} \leq   \frac{Ky_{diff}^{max}}{2\sqrt{W}}\expect{V  +
 \sum_{n=1}^NU_n(t)
 + \sum_{m\in\tilde{\script{M}}} |Z_m(t)| + \sum_{l=1}^L Q_l(t)}   + c\sum_{k=1}^K\expect{T_k(t)} 
 \end{equation} 
where $c$ is a constant that depends on the maximum possible change in queue backlogs
over one slot.    The random variable $T_k(t)$ can be viewed as a 
sum of $W$ geometric random variables (each with mean $K/\theta$), with the possible
exception when $t$ is small and 
some of the past $W$ samples occur during the initialization time $\tau \in \{-WK, -WK + 1, \ldots, -1\}$. 
Therefore, for all $t$ and all $k$ we have: 
\[ \expect{T_k(t)} \leq WK/\theta + WK \]
Then inequality (\ref{eq:finally-25}) satisfies the condition
(\ref{eq:approx-assumption}) from Assumption 3 with: 
\[ \epsilon_V = \epsilon_U = \epsilon_Z = \epsilon_Q = \frac{Ky_{diff}^{max}}{2\sqrt{W}}  \: \: , \: \: C \defequiv c[WK^2/\theta + WK^2] 
 \]
This completes the proof
of Theorem \ref{thm:2}.

\section*{Appendix D --- Proof of Lemma \ref{lem:lln}} 

\begin{proof} 
We have: 
\begin{eqnarray*}
\expect{|Y^{(W)} - \overline{Y}|}^2 \leq \expect{|Y^{(W)} - \overline{Y}|^2} = \frac{\sigma^2}{W} 
\end{eqnarray*}
where $\sigma^2$ is the variance of $Y_1$. It suffices to bound $\sigma^2$ in terms of the 
constants $y_{min}$, $y_{max}$, and  $y_{diff}$. 
We have: 
\begin{eqnarray}
\sigma^2 = Var(Y_1) &=& Var(Y_1 - y_{min}) \nonumber \\
&=& \expect{(Y_1- y_{min})^2} - (\overline{Y}-y_{min})^2 \nonumber \\
&\leq& \expect{(y_{max} - y_{min})(Y_1-y_{min})} - (\overline{Y} - y_{min})^2 \label{eq:ymin}  \\
&=& (\overline{Y} - y_{min})(y_{max} - y_{min} - (\overline{Y} - y_{min})) \nonumber \\
&=& (\overline{Y} - y_{min})(y_{max} - \overline{Y}) \label{eq:appb-final} 
\end{eqnarray}
where (\ref{eq:ymin}) holds because $Y_1-y_{min} \geq 0$. 
To compute the final bound on the expression in (\ref{eq:appb-final}), 
note that $y_{min} \leq \overline{Y} \leq y_{max}$, and  the maximum of 
the function $f(x) = (x - y_{min})(y_{max} - x)$ over the interval $y_{min} \leq x \leq y_{max}$ 
is equal to $(y_{max} - y_{min})^2/4$.   Thus, $\sigma^2 \leq y_{diff}^2/4$.  
\end{proof} 

\section*{Appendix E --- Proof of Theorem \ref{thm:variable-w}}

The proof of Theorem \ref{thm:2} can be followed in the same way, with the exception
that the fixed value $W$ is replaced by the random value $W(t)$ (which may be correlated with queue
states).   Therefore, repeating the proof in Appendix C, the  result of (\ref{eq:finally-25}) translates to: 
\[   \expect{\Omega(\bv{\Theta}(t))} \leq  \frac{Ky_{diff}^{max}}{2}\expect{\frac{V + \sum_n U_n(t) + \sum_m |Z_m(t)| + \sum_l Q_l(t)}{\sqrt{W(t)}}} + c\sum_{k=1}^K \expect{T_k(t)}  \]
Each term $\expect{T_k(t)}$ can be bounded by $\hat{W}(t)K/\theta + W_0K$. 
The final term can thus be bounded as follows: 
\[ c\sum_{k=1}^K \expect{T_k(t)} \leq \hat{W}(t)K^2/\theta + W_0K^2 \]
where $\hat{W}(t) \defequiv (t+1)^{\beta_1}$. 
Define $C_1(t) \defequiv \hat{W}(t)K^2/\theta + W_0K^2$.

It is not difficult to show that $W_{rand}(t)$ satisfies: 
\[ \lim_{t\rightarrow\infty} \frac{W_{rand}(t)}{t} = \frac{\theta}{K} \: \: \mbox{ with probability 1} \]
However, $\hat{W}(t)$ increases sub-linearly with $t$.  Therefore, because $W(t) \defequiv \min[\hat{W}(t), W_{rand}(t)]$, 
we have: 
\[ \lim_{t\rightarrow\infty} Pr[W(t) \neq \hat{W}(t)] = 0 \]
Furthermore, because $W_{rand}(t)$ is simply the min of $K$  delayed renewal
processes $W_1(t), \ldots, W_K(t)$ (each having i.i.d. geometric inter-arrival times with mean $K/\theta$), we have
by the union bound: 
\[ Pr[W(t) \neq \hat{W}(t)] = Pr\left[\min[W_1(t), \ldots, W_K(t)] <  (t+1)^{\beta_1}\right] \leq  KPr\left[W_1(t) \leq (t+1)^{\beta_1}\right] \]
It follows that: 
\[ \lim_{t\rightarrow\infty} tPr[W(t) \neq \hat{W}(t)] = 0\]

Therefore: 
\begin{eqnarray*}
\frac{Ky_{diff}^{max}}{2}\expect{\frac{V + \sum_n U_n(t) + \sum_m |Z_m(t)| + \sum_l Q_l(t)}{\sqrt{W(t)}}} \leq\\
\frac{Ky_{diff}^{max}}{2\sqrt{\hat{W}(t)}}\expect{V + \sum_n U_n(t) + \sum_m |Z_m(t)| + \sum_l Q_l(t)\left|\right.W(t) = \hat{W}(t)}Pr[\hat{W}(t)=W(t)] \\
+ \frac{Ky_{diff}^{max}}{2}\expect{V + \sum_n U_n(t) + \sum_m |Z_m(t)| + \sum_l Q_l(t)\left|\right.W(t)\neq \hat{W}(t)}Pr[\hat{W}(t)\neq W(t)]
\end{eqnarray*}
where we have used the fact that $W(t) \geq 1$ always. Adding the (non-negative) conditional expectation
 to complete
the first term on the right hand side yields: 
\begin{eqnarray*}
\frac{Ky_{diff}^{max}}{2}\expect{\frac{V + \sum_n U_n(t) + \sum_m |Z_m(t)| + \sum_l Q_l(t)}{\sqrt{W(t)}}} \leq\\
\frac{Ky_{diff}^{max}}{2\sqrt{\hat{W}(t)}}\expect{V + \sum_n U_n(t) + \sum_m |Z_m(t)| + \sum_l Q_l(t)} \\
+ \frac{Ky_{diff}^{max}}{2}\expect{V + \sum_n U_n(t) + \sum_m |Z_m(t)| + \sum_l Q_l(t)\left|\right.W(t)\neq \hat{W}(t)}Pr[\hat{W}(t)\neq W(t)] \\
\leq \frac{Ky_{diff}^{max}}{2\sqrt{\hat{W}(t)}}\expect{V + \sum_n U_n(t) + \sum_m |Z_m(t)| + \sum_l Q_l(t)} \\
+  \frac{Ky_{diff}^{max} c_0t}{2} Pr[\hat{W}(t) \neq W(t)] 
\end{eqnarray*}
where $c_0$ is a constant that is proportional to the maximum change in any queue over one slot. 
Because $t Pr[\hat{W}(t) \neq W(t)] \rightarrow 0$ as $t\rightarrow\infty$, there exists a time $t_0$ such that for all $t \geq t_0$
we have:
\[ \frac{Ky_{diff}^{max}c_0t}{2}Pr[\hat{W}(t) \neq W(t)] \leq 1 \]
  We can now define $C(t) \defequiv C_1(t) + 1$ for use
in Theorem \ref{thm:variable-v} (note that $C(t) \leq O((t-t_0+1)^{\beta_1})$). Further define: 
\[ \epsilon_x(t) \defequiv \frac{Ky_{diff}^{max}}{2\sqrt{\hat{W}(t)}} \]
for $x \in \{V, U, Z, Q\}$.  This satisfies the assumptions of Theorem \ref{thm:variable-v}, proving the result.

\bibliographystyle{unsrt}
\bibliography{../../latex-mit/bibliography/refs}
\end{document}